\documentclass[11pt, a4paper]{amsart}
\usepackage[latin1]{inputenc}
\usepackage[english]{babel}
\usepackage{amssymb, amsmath, amsthm}
\usepackage[foot]{amsaddr}
\usepackage{hyperref}
\usepackage[twoside=false]{geometry}
\usepackage{verbatim}
\usepackage{calrsfs}
\usepackage{todonotes}
\usepackage{cleveref}
\usepackage{graphicx}
\usepackage{blkarray}
\usepackage{tikz}
\usetikzlibrary{matrix,arrows,calc,cd}
\tikzcdset{every label/.append style = {font = \small}}
\usepackage{package}
\newcommand{\LL}{\mathcal{L}}
\DeclareMathOperator{\SI}{SI}
\makeatletter
\newcommand{\subalign}[1]{%
  \vcenter{%
    \Let@ \restore@math@cr \default@tag
    \baselineskip\fontdimen10 \scriptfont\tw@
    \advance\baselineskip\fontdimen12 \scriptfont\tw@
    \lineskip\thr@@\fontdimen8 \scriptfont\thr@@
    \lineskiplimit\lineskip
    \ialign{\hfil$\m@th\scriptstyle##$&$\m@th\scriptstyle{}##$\hfil\crcr
      #1\crcr
    }%
  }%
}
\makeatother
\newlength{\tinyspace}
\title{A Gelfand--MacPherson correspondence for quiver moduli}
\author{Hans Franzen}
\date{}
\address{Universit\"at Paderborn, Warburger Str.\ 100, 33098 Paderborn, Germany}
\email{hans.franzen{\hspace{\tinyspace}@\hspace{\tinyspace}}math.upb.de}
\begin{document}
	
	\begin{abstract}
		We show that a semi-stable moduli space of representations of an acyclic quiver can be identified with two GIT quotients by reductive groups. One of a quiver Grassmannian of a projective representation, the other of a quiver Grassmannian of an injective representation. This recovers as special cases the classical Gelfand--MacPherson correspondence and its generalization by Hu and Kim to bipartite quivers, as well as the Zelevinsky map for a quiver of Dynkin type $A$ with the linear orientation.
	\end{abstract}

	\maketitle
	
	\section{Introduction} \label{s:intro}
	
	In \cite{GM:82}, Gelfand and MacPherson use Grassmannians to construct point configurations in projective spaces. They provide a bijection between certain $\GL_n(\C)$-orbits on $(\P^{n-1})^m$ and certain $(\C^\times)^m$-orbits in the Grassmannian $\Gr^n(\C^m)$ of $n$-codimensional subspaces of $\C^m$. Kapranov \cite{Kapranov:93} formulated this correspondence, which he calls the Gelfand--MacPherson correspondence, in terms of GIT quotients. For a tuple of positive integers $a = (a_1,\ldots,a_m)$, consider the unique $\SL_n$-linearization of the ample line bundle $\OO(a) = \OO(a_1,\ldots,a_m)$ on $(\P^{n-1})^m$. It is shown that there exists an ample linearization $\mathcal{L}$ of the action of the torus $T = \{(t_1,\ldots,t_m) \in (\C^\times)^m \mid t_1\ldots t_m = 1\}$ on the Grassmannian $\Gr^n(\C^m)$ for which there is an isomorphism of varieties
	\[
		\Gr^n(\C^m)^{\mathcal{L}\hypsst}/\!\!/T \xto{}{\cong} ((\P^{n-1})^m)^{\OO(a)\hypsst}/\!\!/\SL_n.
	\]
	This isomorphism is provided by an identification of both quotients with a semi-stable GIT quotient of the space of $n \times m$-matrices by the conjugation action of the subgroup $G$ of $\GL_n \times \GL_m$ of all $(g,\diag(t_1,\ldots,t_m))$ such that $\det(g)t_1^{-1}\ldots t_m^{-1} = 1$.
	If the $n \times m$-matrix has rank $n$ then its kernel is an $n$-codimensional subspace of $\C^m$, while if all its columns are non-zero, we obtain a configuration of $m$ points in $\P^{n-1}$.
	
	The semi-stable quotient of the space of $n \times m$-matrices by the group $G$ is a moduli space of semi-stable quiver representations. The quiver is the $m$-subspace quiver (see \Cref{e:ex1.1}).
	This observation is generalized by Hu and Kim in \cite{HK:13}. They provide a Gelfand--MacPherson correspondence for bipartite quivers (they call them quivers of fence type), i.e. quivers in which every vertex is either a source or a sink, and stability parameters which are positive at the sources and negative at the sinks.
	
	In this paper we provide a generalization of Hu and Kim's Gelfand--MacPherson correspondence to any acyclic quiver and arbitrary stability conditions. The idea is as follows. If $M$ is a representation of a quiver $Q$ of dimension vector $\alpha$, semi-stable with respect to a given stability parameter $\theta$, then we find a surjective homomorphism $\phi: P \to M$ from a projective representation; here $P$ depends only on $\theta$, while the homomorphism depends on $M$. This makes $M$ a quotient of $P$ and thus gives a morphism from the locus of all $\theta$-semi-stable representations to the quiver Grassmannian $\Gr^\alpha(P)$ of quotient representations whose dimension vector is $\alpha$. With the same arguments as by Bongartz and Huisgen-Zimmermann in \cite{BHZ:01}, we see that orbits of the group $G(\alpha)$ on the semi-stable locus $R(Q,\alpha)^{\theta\hypsst}$ of the representation variety correspond bijectively to orbits of $\Aut(P)$ on $\Gr^\alpha(P)$. As the automorphism group of $P$ is in general non-reductive and rather complicated, we will replace it by a reductive subgroup and show that we are provided with an isomorphism between the $\theta$-semi-stable moduli space $M^{\theta\hypsst}(Q,\alpha)$ with a semi-stable GIT quotient of $\Gr^\alpha(P)$ by this reductive group -- here semi-stability is with respect to an equivariant ample line bundle determined by $\theta$.
	
	The same construction can be carried out dually. We find an injective homomorphism $\psi: M \to I$ to an injective representation. This yields a map from the semi-stable locus to the quiver Grassmannian $\Gr_\alpha(I)$ of subrepresentations of dimension vector $\alpha$. Again, we find an action of a reductive group on $\Gr_\alpha(I)$ such that the semi-stable quotient with respect to a certain equivariant ample line bundle is isomorphic to the semi-stable moduli space.
	
	These two dual constructions yield two quotients of quiver Grassmannians which are both isomorphic to the same moduli space of quiver representations. In the case of a bipartite quiver and a stability parameter as considered in \cite{HK:13}, we recover the Gelfand--MacPherson correspondence of Hu and Kim.
	
	The case of the trivial stability condition gives, although the semi-stable quiver moduli space is a point, interesting morphisms $\Gr^\alpha(P) \ot R(Q,\alpha) \to \Gr_\alpha(I)$. The group which acts on all these three varieties is $G(\alpha)$ and both morphisms are equivariant with respect to $G(\alpha)$. If $Q$ is the quiver of Dynkin type $A_n$ with the linear orientation, then the map $R(Q,\alpha) \to \Gr_\alpha(I)$ is the so-called Zelevinsky map, see \cite{Zelevinsky:85, LM:98} and \Cref{e:ex2.2}.
	
	We prove the Gelfand--MacPherson correspondence by identifying the quiver Grassmannians $\Gr^\alpha(P)$ and $\Gr_\alpha(I)$ with non-commutative Hilbert schemes. These are special cases of moduli spaces of representations of framed quivers. On these, we find actions of certain reductive groups and ample line bundles which correspond to the given stability parameter $\theta$. We state the main result, i.e.\ the correspondence, in two versions, one for non-commutative Hilbert schemes in \Cref{t:correspondence1} and one for quiver Grassmannians in \Cref{t:correspondence2}.
	
	The paper is organized as follows. In Sections \ref{s:rep_th}, \ref{s:quiver_moduli}, and \ref{s:quiver_gr} we recall basic facts on representations of quivers, quiver moduli, and quiver Grassmannians and we fix notation. In \Cref{s:map_to_Gr}, we describe the constructions of the maps from the semi-stable locus to the quiver Grassmannians. We identify the latter with non-commutative Hilbert schemes in \Cref{s:NCHS} and define an action of a reductive group on those. \Cref{s:descent} is concerned with line bundles on non-commutative Hilbert schemes and the analysis of stability. In \Cref{s:maps_to_NCHS} we define morphisms from the semi-stable locus to non-commutative Hilbert schemes and show that they give isomorphisms on the level of GIT quotients. In \Cref{s:correspondence} we formulate the (two versions of the) main result and give the proof.
	We illustrate our constructions with two running examples.
	
	\subsection*{Acknowledgements}
	We would like to thank Pieter Belmans, Victoria Hoskins, and Markus Reineke for inspiring discussions on the subject and for helpful remarks.

	\section{Representations of quivers} \label{s:rep_th}
	
	We work over a fixed algebraically closed field $k$ throughout. Unless otherwise mentioned, all vector spaces, varieties, and so on are defined over $k$.
	
	Let $Q$ be a quiver. It consists of two finite sets $Q_0$, the set of vertices of $Q$, and $Q_1$, the set of arrows of $Q$, together with two maps $s,t: Q_1 \to Q_0$, assigning to an arrow its source and target vertex. We always assume the quiver $Q$ to be acyclic, i.e.\ it contains no oriented cycles. 
	
	A representation $M$ of $Q$ consists of a collection of vector spaces $M_i$ for $i \in Q_0$ and linear maps $M_a: M_{s(a)} \to M_{t(a)}$ for every $a \in Q_1$. A representation $M$ is called finite-dimensional, if $\bigoplus_{i \in Q_0} M_i$ is a finite-dimensional vector space. A homomorphism $f: M \to N$ is a collection of linear maps $f_i: M_i \to N_i$ such that $N_a f_{s(a)} = f_{t(a)} M_a$ holds for every $a \in Q_1$. We obtain the category of representations of $Q$ and the full subcategory of finite-dimensional representations. For a finite-dimensional representation $M$, the tuple $\dimvect M := (\dim M_i)_i \in \smash{\Z_{\geq 0}^{Q_0}}$ is called the dimension vector of $Q$. 
	
	The path algebra $kQ$ of $Q$ is the vector space spanned by all (oriented) paths in $Q$; this includes a path of length zero $\epsilon_i$ at every vertex $i \in Q_0$. Let $Q_*$ be the set of all paths of $Q$. The multiplication of $kQ$ is on basis elements, i.e.\ on paths, given by
	\[
		q\cdot p = \begin{cases} qp & \text{if } t(p) = s(q) \\ 0 & \text{otherwise,} \end{cases}
	\]
	so by concatenation of paths, if possible. The category of left $kQ$-modules is equivalent to the category of representations of $Q$ by assigning to a left $kQ$-module $M$ the representation consisting of $M_i = \epsilon_iM$ and $M_a(v) = av$ for $v = \epsilon_{s(a)}v \in M_{s(a)}$. In particular we see that the category of representations is abelian and the category of finite-dimensional representations is a full abelian finite length subcategory. We will in the following not always make a distinction between left $kQ$-modules and representations of $Q$.
	
	Let $A = kQ$. For $i \in Q_0$ the left $A$-module $P(i) = A\epsilon_i$ is projective. Under the above equivalence it corresponds to the representation with 
	\[
		P(i)_j = \epsilon_jA\epsilon_i = \bigoplus_{p \in Q_*(i,j)} kp,
	\]
	where $Q_*(i,j)$ denotes the set of paths from $i$ to $j$ in $Q$. There is a natural isomorphism of vector spaces $M_i \cong \Hom_A(P(i),M)$.
	As $Q$ is acyclic, $P(i)$ is finite-dimensional. 
	
	For a right $A$-module $N$ let $D(N) = \Hom(N,k)$, regarded as a left $A$-module. The left $A$-module $I(i) := D(\epsilon_iA)$ is injective. Concretely, 
	\[
		I(i)_j = (\epsilon_iA\epsilon_j)^* = \bigoplus_{q \in Q_*(j,i)} kq^*,
	\]
	where $q^*$ denotes the dual basis element. There is a natural isomorphism of vector spaces $M_i^* \cong \Hom_A(M,I(i))$.
	
	For every left $A$-module $M$, there exist two ``canonical'' resolutions. The canonical projective resolution is given by
	\[\begin{tikzcd}
		0 & {\displaystyle\bigoplus_{a \in Q_1} P(t(a)) \otimes M_{s(a)}} & {\overbrace{\displaystyle\bigoplus_{i \in Q_0} P(i) \otimes M_i}^{=: P}} & M & 0 \\
		0 & {\displaystyle\bigoplus_{a \in Q_1} A\epsilon_{t(a)} \otimes \epsilon_{s(a)}M} & {\displaystyle\bigoplus_{i \in Q_0} A\epsilon_i \otimes \epsilon_iM} & M & 0 \\[-2em]
		& {q \otimes v} & {qa \otimes v - q \otimes av} \\[-2em]
		&& {p \otimes v} & pv
		\arrow[from=1-1, to=1-2]
		\arrow[from=1-2, to=1-3]
		\arrow["\phi", from=1-3, to=1-4]
		\arrow[from=1-4, to=1-5]
		\arrow[from=2-1, to=2-2]
		\arrow[from=2-2, to=2-3]
		\arrow[from=2-3, to=2-4]
		\arrow[from=2-4, to=2-5]
		\arrow[maps to, from=3-2, to=3-3]
		\arrow[maps to, from=4-3, to=4-4]
		\arrow[equals, equals, from=1-2, to=2-2]
		\arrow[equals, equals, from=1-3, to=2-3]
		\arrow[equals, equals, from=1-4, to=2-4]
	\end{tikzcd}\]
	the tensor products in the above diagram are over $k$. The map $\phi$ corresponds under the isomorphisms
	\[
		\Hom_A(P,M) \cong \bigoplus_{i \in Q_0} \Hom_k(M_i,\Hom_A(P(i),M)) \cong \bigoplus_{i \in Q_0} \End_k(M_i)
	\]
	to the tuple $(\id_{M_i})_i$. Similarly, the map $\psi$ in the canonical injective resolution
	\[\begin{tikzcd}
		0 & M & {\underbrace{\displaystyle\bigoplus_{i \in Q_0} I(i) \otimes M_i}_{=: I}} & {\displaystyle\bigoplus_{a \in Q_1} I(s(a)) \otimes M_{t(a)}} & 0
		\arrow[from=1-1, to=1-2]
		\arrow["\psi", from=1-2, to=1-3]
		\arrow[from=1-3, to=1-4]
		\arrow[from=1-4, to=1-5]
	\end{tikzcd}\]
	maps $v \in M$ to $\psi(v) = \sum_{q \in Q_*} q^* \otimes qv$. It corresponds under the isomorphisms
	\[
		\Hom_A(M,I) \cong \bigoplus_{i \in Q_0} \Hom_A(M,I(i)) \otimes M_i \cong \bigoplus_{i \in Q_0} M_i^* \otimes M_i \cong \bigoplus_{i \in Q_0} \End_k(M_i)
	\]
	to the tuple $(\id_{M_i})_i$.
	
	Note that a consequence of the canonical resolutions is that the category of finite-dimensional representations of $Q$ is hereditary, which means that $\Ext_A^i(M,N) = 0$ for all $i \geq 2$ and all representations $M$ and $N$.
	
	The numerical Euler form of the quiver $Q$ is the (non-symmetric) bilinear form $\langle \blank,\blank \rangle: \Z^{Q_0} \times \Z^{Q_0} \to \Z$ defined by
	\[
		\langle \alpha,\beta \rangle = \sum_{i \in Q_0} \alpha_i\beta_i - \sum_{a \in Q_1} \alpha_{s(a)}\beta_{t(a)}.
	\]
	The homological Euler form of the category of representations of $Q$ depends only on the dimension vector of $Q$ and can be computed using the numerical Euler form:
	\[
		\dim \Hom_A(M,N) - \dim \Ext_A^1(M,N) = \langle \dimvect M,\dimvect N \rangle.
	\]

	\section{Quiver moduli} \label{s:quiver_moduli}
	
	Let $Q$ be an acyclic quiver. In the following, all representations are assumed to be finite-dimensional. Let $\alpha \in \smash{\Z_{\geq 0}^{Q_0}}$. Fix vector spaces $V_i$ of dimension $\dim V_i = \alpha_i$. We consider the vector space
	\[
		R(Q,\alpha) := \bigoplus_{a \in Q_1} \Hom(V_{s(a)},V_{t(a)}).
	\]
	An element of $R(Q,\alpha)$ is a representation of $Q$ on the fixed set of vector spaces $(V_i)_i$. On $R(Q,\alpha)$, we have a linear left action of the algebraic group
	\[
		G(\alpha) := \prod_{i \in Q_0} \GL(V_i),
	\]
	where $g = (g_i)_i \in G(\alpha)$ acts on $M = (M_a)_a \in R(Q,\alpha)$ by $g\cdot M = (\smash{g_{t(a)}}M_a\smash{g_{s(a)}^{-1}})_a$. Note that two elements of $R(Q,\alpha)$ are isomorphic as representations of $Q$ if and only if they lie in the same $G(\alpha)$-orbit. The (closed normal) subgroup $\Delta := k^\times \cdot \id$ acts trivially on $R(Q,\alpha)$. We thus obtain an action of the group $PG(\alpha) := G(\alpha)/\Delta$.
	
	In order to construct moduli spaces, we introduce a stability parameter. Let $\theta \in \Z^{Q_0}$. For $\alpha \in \Z^{Q_0}$ we define $\theta(\alpha) := \sum_{i \in Q_0} \theta_i\alpha_i$ and for a representation $M$ of $Q$ define $\theta(M) := \theta(\dimvect M)$. 
	
	\begin{defn} \label{d:stability}
		Let $M \in R(Q,\alpha)$. 
		\begin{enumerate}
			\item $M$ is called $\theta$-\emph{semi-stable} if $\theta(M) = 0$ and $\theta(M') \leq 0$ for all non-zero proper subrepresentations $M'$ of $M$.
			\item $M$ is called $\theta$-\emph{stable} if $\theta(M) = 0$ and $\theta(M') < 0$ for all non-zero proper subrepresentations $M'$ of $M$.
		\end{enumerate}
	\end{defn}
	We denote by $R(Q,\alpha)^{\theta\hypst} \sub R(Q,\alpha)^{\theta\hypsst}$ the loci of $\theta$-stable and $\theta$-semi-stable representations, respectively. These are Zariski open $G(\alpha)$-invariant subsets of $R(Q,\alpha)$. Note that these may be empty. There are examples where $R(Q,\alpha)^{\theta\hypsst}$ is empty and also examples where $R(Q,\alpha)^{\theta\hypst}$ is empty but $R(Q,\alpha)^{\theta\hypsst}$ is not. Observe also that $R(Q,\alpha)^{0\hypsst} = R(Q,\alpha)$.

	\begin{rem}
		The above definition of (semi\nobreakdash-)stability differs from King's original definition in \cite{King:94} by a minus sign. 
	\end{rem}
	
	The full subcategory of $\theta$-semi-stable representations is an abelian subcategory of the category of representations of $Q$, see \cite{Reineke:03}. The $\theta$-stable representations are the simple objects in this category. Schur's lemma implies that the automorphism group of a stable representation is $k^\times$. This shows that $PG(\alpha)$ acts freely on $R(Q,\alpha)^{\theta\hypst}$.

	The notion of $\theta$-(semi\nobreakdash-)stability agrees with Mumford's \cite{GIT:94} notion of (semi\nobreakdash-)stability with respect to an ample equivariant line bundle. The line bundle is given as follows.
	We define a character $\chi_\theta: G(\alpha) \to k^\times$ associated with $\theta$ by
	$$
		\chi_\theta(g) = \prod_i \det(g_i)^{-\theta_i}.
	$$
	Note the minus sign in the exponent; it occurs because of our opposite sign convention in \Cref{d:stability}.
	As $\theta(\alpha) = 0$, 
	the character $\chi_\theta$ descends to a character $\bar{\chi}_\theta: PG(\alpha) \to k^\times$. Let $L(\theta) := L(\bar{\chi}_\theta)$ be the trivial line bundle on $R(Q,\alpha)$ equipped with the $PG(\alpha)$-linearization given by the character $\bar{\chi}_\theta$. King shows in \cite{King:94} that $M$ is $\theta$-semi-stable if and only if it is semi-stable with respect to $L(\theta)$ and $M$ is $\theta$-stable if and only if it is properly stable with respect to $L(\theta)$. 
	
	For a character $\chi$ of $G(\alpha)$ we define
	\[
		\SI(Q,\alpha)_\chi := \{ f \in k[R(Q,\alpha)] \mid f(g^{-1}\cdot M) = \chi(g) f(M) \text{ (all $g \in G(\alpha)$, $M \in R(Q,\alpha)$)} \}.
	\]
	An element of $\SI(Q,\alpha)_\chi$ is called a $\chi$-semi-invariant function. For all $n \geq 0$ holds
	\[
		H^0(R(Q,\alpha),L(\theta)^{\otimes n})^{PG(\alpha)} = \SI(Q,\alpha)_{-n\chi_\theta}.
	\]
	We define
	\[
		M^{\theta\hypsst}(Q,\alpha) := R(Q,\alpha)^{L(\theta)\hypsst}/\!\!/ PG(\alpha) = \Proj\Big( \bigoplus_{n \geq 0} \SI(Q,\alpha)_{-n\chi_\theta} \Big)
	\]
	and let $\pi: R(Q,\alpha)^{\theta\hypsst} \to M^{\theta\hypsst}(Q,\alpha)$ be the quotient morphism. Moreover, we set
	\[
		M^{\theta\hypst}(Q,\alpha) := R(Q,\alpha)^{L(\theta)\hypst}/ PG(\alpha) = \pi(R(Q,\alpha)^{\theta\hypst}).
	\]
	We denote the restriction of the quotient map also by $\pi: R(Q,\alpha)^{\theta\hypst} \to M^{\theta\hypst}(Q,\alpha)$. 
	
	\begin{defn}
		We call $M^{\theta\hypsst}(Q,\alpha)$ and $M^{\theta\hypst}(Q,\alpha)$ the $\theta$-\emph{semi-stable} and the $\theta$-\emph{stable moduli space} of $Q$ of dimension vector $\alpha$, respectively.
	\end{defn}
	
	A result by Le Bruyn and Procesi \cite{LP:90} implies that, for an acyclic quiver, there are no non-constant $G(\alpha)$-invariant regular functions on $R(Q,\alpha)$. Therefore $M^{\theta\hypsst}(Q,\alpha)$ is projective. The quotient $\pi: R(Q,\alpha)^{\theta\hypst} \to M^{\theta\hypst}(Q,\alpha)$ is a geometric $PG(\alpha)$-quotient with trivial isotropy groups on the total space. If $\operatorname{char}(k) = 0$ then Luna's \'{e}tale slice theorem \cite{Luna:73} implies that the stable quotient map is a principal left $PG(\alpha)$-bundle in the \'{e}tale topology.
	
	Let us look at two examples, which will be the running examples for the text.
	
	\begin{ex} \label{e:ex1.1}
		Let $m \geq 1$ and $n \geq 2$ be integers and let $Q$ be the $m$-subspace quiver:
		
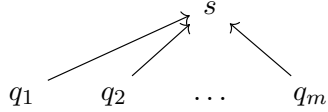
\[\begin{tikzcd}
			&[-1em]&[-1em] s &[-1em]\\
			q_1 & q_2 & \ldots & q_m
			\arrow[from=2-1, to=1-3]
			\arrow[from=2-2, to=1-3]
			\arrow[from=2-4, to=1-3]
		\end{tikzcd}\]
		Consider the dimension vector $\alpha = (1,\ldots,1,n)$, i.e. $1$ at every source and $n$ at the sink. A representation of $Q$ of dimension vector $\alpha$ is a tuple $(v_1,\ldots,v_m)$ of vectors $v_i \in k^n$. The group $G(\alpha)$ is $(k ^\times)^m \times \GL_n(k)$ which acts on $R(Q,\alpha) = (k^n)^m$ by
		\[
			(t_1,\ldots,t_m,g) \cdot (v_1,\ldots,v_m) = (t_1^{-1}gv_1,\ldots,t_m^{-1}gv_m)
		\]
		Let $a = (a_1,\ldots,a_m)$ be a sequence of positive integers and let $|a| = a_1+\ldots+a_m$. Define $\theta = (na_1,\ldots,na_m,-|a|)$. For $I \sub \{1,\ldots,m\}$ let $|a|_I := \sum_{i \in I} a_i$. A representation is $\theta$-semi-stable (resp.\ $\theta$-stable) if and only if $v_i \neq 0$ for all $i=1,\ldots,m$ and $\sum_{i \in I} k\cdot v_i = k^n$ for all subsets $I \sub \{1,\ldots,m\}$ for which $|a|_I > |a|/2$ (resp.\ $|a|_I \geq |a|/2$). 
	\end{ex}

	\begin{ex} \label{e:ex2.1}
		Let $Q$ be the following orientation of the Dynkin diagram of type $A_n$:
		\[
			1 \to 2 \to \ldots \to n
		\]
		Let $\alpha = (\alpha_1,\ldots,\alpha_n)$ be an arbitrary dimension vector. A representation of $Q$ of dimension vector $\alpha$ is a tuple $(M_1,\ldots,M_{n-1})$ of matrices $M_i \in M_{\alpha_{i+1}\times \alpha_i}(k)$. The group $G(\alpha) = \GL_{\alpha_1}(k) \times \ldots \times \GL_{\alpha_n}(k)$ acts on $R(Q,\alpha)$ by
		\[
			(g_1,\ldots,g_n)\cdot (M_1,\ldots,M_{n-1}) = (g_2M_1g_1^{-1},\ldots,g_nM_{n-1}g_{n-1}^{-1}).
		\]
		Let $\theta = 0$. Then every representation is $\theta$-semi-stable. A representation is $\theta$-stable if and only if it is simple; in this case $\alpha = e_i := (0,\ldots,1,\ldots,0)$ for some vertex $i \in \{1,\ldots,n\}$. The $\theta$-semi-stable moduli space
		\[
			M^{\theta\hypsst}(Q,\alpha) =: M^{\operatorname{ssimp}}(Q,\alpha)
		\]
		consists of a single point, because the semi-simple representation of dimension vector $\alpha$ is contained in the $G(\alpha)$-orbit closure of every other representation. The $\theta$-stable moduli space is either empty (if $\alpha$ is not a simple root) or a point (if $\alpha = e_i$ for some vertex $i$).
		
		\noindent
		Although the moduli spaces in this case are trivial, the Gelfand--MacPherson correspondence for them will still be interesting.
	\end{ex}

	\section{Quiver Grassmannians} \label{s:quiver_gr}
	
	Let $M$ be a representation of an acyclic quiver $Q$ of dimension vector $\alpha := \dimvect M$ and let $\beta \in \smash{\Z_{\geq 0}^{Q_0}}$ such that $\beta_i \leq \alpha_i$ for all $i \in Q_0$. Inside the product $\prod_{i \in Q_0} \Gr_{\beta_i}(M_i)$ we consider the Zariski closed subset
	\[
		\Gr_\beta(M) := \Big\{ U = (U_i)_{i \in Q_0} \in \prod_{i \in Q_0} \Gr_{\beta_i}(M_i) \mid M_a(U_{s(a)}) \sub U_{t(a)} \text{ (all $a \in Q_1$)} \Big\}.
	\]
	It can be interpreted as the zero locus of a section of a vector bundle on $\prod_{i \in Q_0} \Gr_{\beta_i}(M_i)$. Therefore $\Gr_\beta(M)$ has a natural structure of a closed subscheme of $\prod_{i \in Q_0} \Gr_{\beta_i}(M_i)$; it is in particular projective. Its set of $k$-valued points is the set of subrepresentations of $M$ whose dimension vector is $\beta$. 
	
	\begin{defn}
		We call $\Gr_\beta(M)$ the \emph{quiver Grassmannian of} $\beta$-\emph{dimensional subrepresentations} of $M$. We also define $\Gr^\beta(M) := \Gr_{\alpha-\beta}(M)$ and call it the \emph{quiver Grassmannian of} $\beta$-\emph{dimensional quotients} of $M$.
	\end{defn}
	
	For an arbitrary representation $M$, no favorable geometric properties of quiver Grassmannians may be expected. In fact every closed subscheme of a projective space over $k$ arises as a quiver Grassmannian. This result was proved, in various levels of generality, in several ways by several authors, e.g. \cite{Huisgen-Zimmermann:07, Reineke:13, Hille:15, Ringel:18}.
	
	However, if $M$ is rigid, which means $\Ext_{kQ}^1(M,M) = 0$, then $\Gr_\beta(M)$ is smooth \cite{CR:08} and irreducible \cite{Wolf:09}. Over $k = \C$ it has no odd cohomology \cite{Nakajima:11, Qin:12} and the cycle map is an isomorphism \cite{CEFR:21}. The dimension of $\Gr_\beta(M)$ is $\langle \beta,\alpha-\beta \rangle$.

	\section{A map to quiver Grassmannians} \label{s:map_to_Gr}
	
	Let $Q$ be an acyclic quiver, $A = kQ$ the path algebra, $\alpha \in \smash{\Z_{\geq 0}^{Q_0}}$ a dimension vector, and $\theta \in \Z^{Q_0}$ a stability parameter. Fix vector spaces $V_i$ of dimension $\alpha_i$ for every $i \in Q_0$. We define 
	\begin{align*}
		Q_0^+:= Q^+_{0,\theta} &:= \{ i \in Q_0 \mid \theta_i \geq 0 \} \\
		Q_0^- := Q^-_{0,\theta} &:= \{ i \in Q_0 \mid \theta_i \leq 0 \}.
	\end{align*}
	
	\begin{rem}
		All objects decorated with a $+$ or a $-$ which we will introduce in the following depend on $\theta$. We will mostly suppress this dependency in the notation.
	\end{rem}

	Define a projective representation $P^+_\theta$ and an injective representation $I^-_\theta$ as follows:
	\begin{align*}
		P^+ := P^+_\theta &:= \bigoplus_{i \in Q_0^+} P(i) \otimes V_i &
		I^- := I^-_\theta &= \bigoplus_{i \in Q_0^-} I(i) \otimes V_i \\
		&= \bigoplus_{\substack{p \in Q_*\\ s(p) \in Q_0^+}} kp \otimes V_{s(p)} &
		&= \bigoplus_{\substack{q \in Q_*\\ t(q) \in Q_0^-}} kq^* \otimes V_{t(q)}
	\end{align*}
	
	Let $M \in R(Q,\alpha)$, consisting of linear maps $M_a: V_{s(a)} \to V_{t(a)}$ for every $a \in Q_1$. As in \Cref{s:rep_th}, we see that
	\begin{align*}
		\Hom_A(P^+,M) &\cong \bigoplus_{i \in Q_0^+} \End_k(V_i) &
		\Hom_A(M,I^-) &\cong \bigoplus_{i \in Q_0^-} \End_k(V_i).
	\end{align*}
	Let $\phi_M: P^+ \to M$ and $\psi_M: M \to I^-$ be the two homomorphisms which correspond to the tuples $(\id_{V_i})_{i \in Q_0^\pm}$. Concretely, $\phi_M$ applied to an element $p \otimes v$ with $v \in V_{s(p)}$ is
	\[
		\phi_M(p \otimes v) = pv = M_p(v)
	\]
	while $\psi_M$ applied to $v \in V_i$ gives
	\[
		\psi_M(v) = \sum_{\subalign{q &\in Q_*\\ t(q) &\in Q_0^-}} q^* \otimes qv = \sum_{\subalign{q &\in Q_*\\ t(q) &\in Q_0^-\\ s(q) &= i}} q^* \otimes M_q(v).
	\]
	
	\begin{lem} \label{l:semi-stable_surjective_injective}
		Let $M \in R(Q,\alpha)$ be $\theta$-semi-stable. Then
		\begin{enumerate}
			\item $\phi_M$ is surjective,
			\item $\psi_M$ is injective.
		\end{enumerate}
	\end{lem}
	
	\begin{proof}
		\leavevmode
		\begin{enumerate}
			\item Assume that $\phi_M$ were not surjective. Let $C := \coker \phi_M$ and let $j \in Q_0$ such that $C_j \neq 0$. Then $j \in Q_0 \setminus Q_0^+$. Assume that $j$ is minimal in the sense that $C_i = 0$ for all $i \in Q_0$ for which there exists a path of positive length $i \to j$. Then there exists a surjective homomorphism $C \to S(j)$, which yields a surjective homomorphism $M \to S(j)$. But $\theta(S(j)) = \theta_j < 0$ 
			which contradicts $\theta$-semi-stability of $M$.
			\item Suppose $\psi_M$ were not injective. Let $K := \ker \psi_M$ and let $j \in Q_0$ be a maximal vertex such that $K_j \neq 0$. Then $j \in Q_0 \setminus Q_0^-$. By maximality there exists an injective homomorphism $S(j) \to K$. But as $\theta(S(j)) = \theta_j > 0$, we obtain a contradition to semi-stability of $M$. \qedhere
		\end{enumerate}
	\end{proof}

	We define maps $\phi: R(Q,\alpha)^{\theta\hypsst} \to \Gr^\alpha(P^+)$ and $\psi: R(Q,\alpha)^{\theta\hypsst} \to \Gr_\alpha(I^-)$ by
	\begin{align*}
		\phi(M) &= \ker(\phi_M) & 
		\psi(M) &= \im(\psi_M).
	\end{align*}
	
	\begin{ex} \label{e:ex1.2}
		Recall the notation of \Cref{e:ex1.1}. A $\theta$-semi-stable representation of dimension vector $\alpha$ is a tuple $(v_1,\ldots,v_m)$ of non-zero vectors in $k^n$ for which $\{v_i \mid i \in I\}$ spans $k^n$, whenever $I \sub \{1,\ldots,m\}$ is a subset such that $|a|_I > |a|/2$.
		
		\noindent
		The projective representation $P^+$ is
		
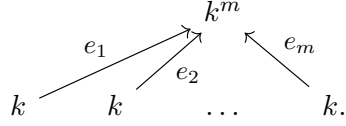
\[\begin{tikzcd}
			&[-.5em]&[-.5em] {k^m} &[-.5em] \\
			k & k & \ldots & k.
			\arrow["{e_1}", from=2-1, to=1-3]
			\arrow["{e_2}"', from=2-2, to=1-3]
			\arrow["{e_m}"', from=2-4, to=1-3]
		\end{tikzcd}\]
		Therefore, the Grassmannian $\Gr^\alpha(P^+)$ identifies with $\Gr^n(k^m)$, the Grassmannian of $n$-codimensional subspaces of $k^m$. Given a $\theta$-semi-stable representation consisting of $(v_1,\ldots,v_m)$, the map $\phi$ assigns to it the kernel of the full rank matrix $(v_1 \mid \ldots \mid v_m)$.
		
		\noindent
		The representation $I^-$ is a direct sum of $n$ copies of the indecomposable injective representation at the sink; explicitly, $I^-$ is
		
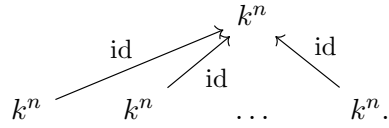
\[\begin{tikzcd}
			&[-.5em]&[-.5em] {k^n} &[-.5em] \\
			{k^n} & {k^n} & \ldots & {k^n}.
			\arrow["\id", from=2-1, to=1-3]
			\arrow["\id"', from=2-2, to=1-3]
			\arrow["\id"', from=2-4, to=1-3]
		\end{tikzcd}\]
		Therefore, $\Gr_\alpha(I^-)$ identifies with $(\P^{n-1})^m$. A $\theta$-semi-stable representation consisting of $(v_1,\ldots,v_m)$ is mapped by $\psi$ to $(\langle v_1 \rangle,\ldots,\langle v_m \rangle)$, the tuple of lines spanned by the non-zero vectors $v_i$.
	\end{ex}
	
	\begin{ex} \label{e:ex2.2}
		Recall \Cref{e:ex2.1}. Write $V_i := k^{\alpha_i}$ and let $V := V_1 \oplus \ldots \oplus V_n$.
		
		\noindent
		Let us treat the injective representation $I^- = I$ first. In this case, we will recover a map considered by Zelevinsky in \cite{Zelevinsky:85}. The representation $I$ is given as
		\[
			V_1 \oplus \ldots \oplus V_n \xto{=: I_{[1,2]}}{\begin{pmatrix} 0 & \id & 0 & \ldots & 0 \\ 0 & 0 & \id & \ddots & \vdots \\ \vdots & \vdots & \ddots & \ddots & 0 \\ 0 & 0 & \ldots & 0 & \id \end{pmatrix}} V_2 \oplus \ldots \oplus V_n \to \ldots \to V_{n-1} \oplus V_n \xto{=: I_{[n-1,n]}}{ \begin{pmatrix} 0 & \id \end{pmatrix}} V_n
		\]
		Denote $E_i := V_1 \oplus \ldots \oplus V_i$. It is the kernel of the composition 
		\[
			I_{[1,i+1]}: V_1 \oplus \ldots \oplus V_n \to V_{i+1} \oplus \ldots \oplus V_n 
		\]		
		in the above diagram. This gives a flag of subspaces $E_1 \sub E_2 \sub \ldots \sub E_n$, which we consider as the standard flag. Let $E'_i := V_{\alpha_{i+1}} \oplus \ldots \oplus V_n$. This yields a flag $E'_1 \supseteq \ldots \supseteq E'_{n-1} \supseteq E'_n$ such that $E_i \oplus E'_i = 0$. 
		
		Let $(U_1,\ldots,U_n) \in \Gr_\alpha(I)$. It consists of $\alpha_i$-dimensional subspaces $U_i$ of $V_i \oplus \ldots \oplus V_n$ such that $I_{[i,i+1]}$ maps $U_i$ into $U_{i+1}$. By taking the inverse image of $U_i$ under $I_{[1,i]}$, we obtain a flag
		\[
			F_1 \sub F_2 \sub \ldots \sub F_n
		\]
		of subspaces of $V$ such that $\dim F_i = \alpha_1 + \ldots + \alpha_i$ and such that $F_i \supseteq E_{i-1}$. We identify the variety of partial flags $F_*$ of subspaces of dimension $\dim F_i = \alpha_1+\ldots+\alpha_i$ with $\GL(V)/H$, where $H$ is the parabolic subgroup of block upper triangular matrices
		\[
			\bordermatrix{ ~ & V_1 & V_2 & \ldots & V_n \cr
				V_1 & * & * & \ldots & * \cr
				V_2 &   & * & \ldots & * \cr
				\ \vdots & & & \ddots & \vdots \cr
				V_n & & & & *
			};
		\]
		it is the parabolic subgroup which fixes the standard flag $E_*$.
		A right coset $gH$ corresponds to the flag stabilized by the parabolic $gHg^{-1}$.
		We obtain a map $\Gr_\alpha(I) \to \GL(V)/H$.
		It can be shown that $\Gr_\alpha(I) \to \GL(V)/H$ is algebraic and restricts to an isomorphism onto the closed subvariety
		\[
			\Omega := \{ F_* \in \GL(V)/H \mid F_i \supseteq E_{i-1} \text{ for $i=2,\ldots,n$} \}.
		\] 
		(Note that the natural scheme structure of $\Omega$ is reduced and that $\Omega$ is irreducible.)
		As the flag $E_*$ is stabilized by $H$, we see that $\Omega$ is $H$-invariant. The subgroup $H$ acts on $\GL(V)/H$ with finitely many orbits, so $\Omega$ must be a closure of an $H$-orbit, i.e.\ a Schubert variety.
		
		\noindent		
		Now to the map $\psi: R(Q,\alpha) \to \Gr_\alpha(I)$. Let $M$ be a representation consisting of $(M_1,\ldots,M_{n-1})$ with $M_i \in \Hom(V_i,V_{i+1})$. The map $\psi_M: M \to I$ in the injective resolution is given by
		\[\begin{tikzcd}
		{V_1\oplus \ldots \oplus V_n} & \ldots & {V_{n-1} \oplus V_n} & {V_n} \\[5em]
		{V_1} & \ldots & {V_{n-1}} & {V_n.}
		\arrow["{\begin{pmatrix} \id \\ M_1 \\ M_2M_1 \\ \vdots \\ M_{n-1}\ldots M_1 \end{pmatrix}}", from=2-1, to=1-1]
		\arrow["{M_1}", from=2-1, to=2-2]
		\arrow[from=2-2, to=2-3]
		\arrow["{M_{n-1}}", from=2-3, to=2-4]
		\arrow["{I_{[1,2]}}", from=1-1, to=1-2]
		\arrow[from=1-2, to=1-3]
		\arrow["{I_{[n-1,n]}}", from=1-3, to=1-4]
		\arrow["{\begin{pmatrix} \id \\ M_{n-1} \end{pmatrix}}"', from=2-3, to=1-3]
		\arrow["\id"', from=2-4, to=1-4]
		\end{tikzcd}\]
		Now take $\im \psi_M \in \Gr_\alpha(I)$, form the associated flag $F_*$ and consider the corresponding right coset $gH \in \GL(V)/H$. It is the right coset of
		\[
			g = g_M = \begin{pmatrix} \id & 0 & \ldots & \ldots & 0 \\ M_1 & \id & \ddots & & \vdots \\ M_2M_1 & M_2 & \ddots & \ddots & \vdots \\ \vdots & \vdots & \ddots & \id & 0 \\ M_{n-1}\ldots M_1 & M_{n-1}\ldots M_2 & \ldots & M_{n-1} & \id  \end{pmatrix}
		\]
		The morphism $\zeta: R(Q,\alpha) \to \Omega$ given by $\zeta(M) =  g_MH$ is the map which Zelevinsky constructs in \cite{Zelevinsky:85}. It is often referred to as the Zelevinsky map. Lakshmibai and Magyar argue in the proof of the main theorem of \cite{LM:98} that $\zeta$ provides an isomorphism from $R(Q,\alpha)$ to the intersection $\Omega \cap O$ with the opposite open Schubert cell
		\[
			O = \{ F_* \in \GL(V)/H \mid F_i \cap E'_i = 0 \text{ for } i=1,\ldots,n \}.
		\]
		
		\noindent
		The representation $P^+ = P$ is
		\[
			V_1 \xto{=: P_{[1,2]}}{\begin{pmatrix} \id \\ 0 \end{pmatrix}} V_1 \oplus V_2 \to \ldots \to V_1 \oplus \ldots \oplus V_{n-1} \xto{=: P_{[n-1,n]}}{\begin{pmatrix} \id & 0 & \ldots & 0 \\ 0 & \id & \ddots & \vdots \\ \vdots & \ddots & \ddots & 0 \\ 0 & \ldots & 0 & \id \\ 0 & \ldots & 0 & 0 \end{pmatrix}} V_1 \oplus \ldots \oplus V_n
		\]
		Setting $P_{[i,n]}: V_1 \oplus \ldots \oplus V_i \to V_1 \oplus \ldots \oplus V_n$, the subspace $E_i$ is the image of $P_{[i,n]}$. For $(U_1,\ldots,U_n) \in \Gr^\alpha(P)$, which is a tuple of subspaces $U_i \sub V_1 \oplus \ldots \oplus V_i$ of codimension $\alpha_i$ such that $P_{[i,i+1]}(U_i) \sub U_{i+1}$, we identify $U_i$ with its image under the map $P_{[i,n]}$. Then $\Gr^\alpha(P)$ identifies with the space of flags
		\[
			U_1 \sub U_2 \sub U_3 \sub \ldots \sub U_n
		\]
		of subspaces of $V$ of dimension $\dim U_i = \alpha_1+\ldots+\alpha_{i-1}$ such that $U_i \sub E_i$. We re-index $F_i := U_{i+1}$ for $i=1,\ldots,n-1$ and obtain a flag $F_1 \sub F_2 \sub F_3 \sub \ldots \sub F_n$. Therefore, we get a map $\Gr^\alpha(P) \to \GL(V)/H$ which can be shown to be algebraic and to induce an isomorphism from $\Gr^\alpha(P)$ onto the closed subvariety 
		\[ 
			\Upsilon := \{ F_* \in \GL(V)/H \mid F_i \sub E_{i+1} \text{ for } i=1,\ldots,n-1 \}.
		\]
		The subvariety is irreducible and $H$-invariant, hence a Schubert variety.
		
		\noindent
		Now to $\phi: R(Q,\alpha) \to \Gr^\alpha(P)$. Let $M$ be a representation of dimension vector $\alpha$ given by $(M_1,\ldots,M_{n-1})$. Then the standard projective resolution  
		\[
			0 \to P^1 \to P \xto{}{\phi_M} M \to 0
		\]
		is given by
		\[\begin{tikzcd}[ampersand replacement=\&]
			{0} \& {V_1} \& \ldots \& {V_1 \oplus \ldots \oplus V_{n-1}} \\[6em]
			{V_1} \& {V_1 \oplus V_2} \& \ldots \& {V_1 \oplus \ldots \oplus V_n} \\
			{V_1} \& {V_2} \& \ldots \& {V_n}
			\arrow[from=1-1, to=2-1]
			\arrow["{\begin{pmatrix} \id \\ -M_1 \end{pmatrix}}", from=1-2, to=2-2]
			\arrow["{\begin{pmatrix} \id & 0 & \ldots & 0 \\ -M_1 & \id & \ddots & \vdots \\ 0 & -M_2 & \ddots & 0 \\ \vdots & \ddots & \ddots & \id \\ 0 & \ldots & 0 & -M_{n-1} \end{pmatrix}}", from=1-4, to=2-4]
			\arrow[from=1-1, to=1-2]
			\arrow[from=1-2, to=1-3]
			\arrow[from=1-3, to=1-4]
			\arrow["\id"', from=2-1, to=3-1]
			\arrow["{(M_1,\id)}", from=2-2, to=3-2]
			\arrow["{(M_{n-1}\ldots M_1,\ldots,M_{n-1},\id)}", from=2-4, to=3-4]
			\arrow["{P_{[1,2]}}", from=2-1, to=2-2]
			\arrow[from=2-2, to=2-3]
			\arrow["{P_{[n-1,n]}}", from=2-3, to=2-4]
			\arrow["{M_1}", from=3-1, to=3-2]
			\arrow[from=3-2, to=3-3]
			\arrow["{M_{n-1}}", from=3-3, to=3-4]
		\end{tikzcd}\]
		For $\ker \phi_M = \im(P^1 \to P)$, we consider the associated flag $F_* \in \Upsilon$ and the corresponding coset $hH \in \GL(V)/H$. It is represented by
		\[
			h = h_M = \begin{pmatrix} \id & 0 & \ldots & \ldots & 0 \\ -M_1 & \id & \ddots & & \vdots \\ 0 & -M_2 & \ddots & \ddots & \vdots \\ \vdots & \ddots & \ddots & \id & 0 \\ 0 & \ldots & 0 & -M_{n-1} & \id  \end{pmatrix}.
		\]
		We obtain a morphism $\eta: R(Q,\alpha) \to \Upsilon$ given by $\eta(M) = h_MH$. We call it the dual Zelevinsky map. In a similar way as in \cite[1.3 Thm., 2.2 Thm.]{LM:98}, we can see that $\eta$ provides an isomorphism onto $\Upsilon \cap O$.
	\end{ex}

	In the next sections, we will show that $\phi$ and $\psi$ induce an isomorphism from the (semi\nobreakdash-)stable moduli space to a GIT quotient of the quiver Grassmannian by a reductive group.
	In order to state and prove this result, it will be convenient to interpret the quiver Grassmannians $\Gr^\alpha(P^+)$ and $\Gr_\alpha(I^-)$ as non-commutative Hilbert schemes. We will do this in the following section.
	
	\section{Non-commutative Hilbert schemes} \label{s:NCHS}
	
	We will first introduce the construction of non-commutative Hilbert schemes in general. We largely follow Reineke's exposition in \cite{Reineke:05}. Note however that our notation differs slightly from his. Let $Q$ be an acyclic quiver and $\beta \in \smash{\Z_{\geq 0}^{Q_0}}$ a dimension vector; later $\beta$ will be the dimension vectors $\alpha^\pm$ considered in \Cref{s:map_to_Gr}. Consider the following two framing constructions.
	
	Let $Q^\beta$ be the quiver consisting of $(Q^\beta)_0 = \{0\} \sqcup Q_0$ and $\beta_i$ arrows $0 \to i$ for every $i \in Q_0$. For any dimension vector $\alpha \in \smash{\Z_{\geq 0}^{Q_0}}$ define the dimension vector $(1,\alpha)$ of $Q^\beta$ by $(1,\alpha)_0 = 1$ and $(1,\alpha)_i = \alpha_i$ for every $i \in Q_0$.
	
	On the other hand, let $Q_\beta$ be the quiver with $(Q_\beta)_0 = Q_0 \sqcup \{\infty\}$ and $\beta_i$ arrows $i \to \infty$ for every $i \in Q_0$. Define the dimension vector $(\alpha,1)$ as before.
	
	Fix vector spaces $V_i$ of dimension $\alpha_i$ and $E_i$ of dimension $\beta_i$. Let 
	\begin{align*}
		P_\beta &= \bigoplus_{i \in Q_0} P(i) \otimes E_i &
		I_\beta &= \bigoplus_{i \in Q_0} I(i) \otimes E_i \\
			&= \bigoplus_{p \in Q_*} kp \otimes E_{s(p)} &
			&= \bigoplus_{q \in Q_*} kq^* \otimes E_{t(p)}
	\end{align*}
	For every $M \in R(Q,\alpha)$ there are isomorphisms 
	\begin{align*}
		\Hom_A(P_\beta,M) &\cong \bigoplus_{i \in Q_0} \Hom_k(E_i,V_i) &
		\Hom(M,I_\beta) &\cong \bigoplus_{i \in Q_0} \Hom_k(V_i,E_i). 
	\end{align*}
	Let $\phi_{M,A}: \smash{P_\beta} \to M$ be the homomorphism corresponding to a tuple $A = (A_i)_i \in \bigoplus_i \Hom_k(E_i,V_i)$ and let $\psi_{M,B}: M \to \smash{I_\beta}$ correspond to $B = (B_i)_i \in \bigoplus_i \Hom_k(V_i,E_i)$. They are given as follows: $\phi_{M,A}$ maps $p \otimes w$ with $w \in E_{s(p)}$ to
	\[
		\phi_{M,A}(p \otimes w) = M_pA_{s(p)}(w),
	\]
	while $\psi_{M,B}$ applied to $v \in V_i$ gives
	\[
		\psi_{M,B}(v) = \sum_{\substack{q \in Q_*\\ s(q) = i}} q^* \otimes B_{t(q)}M_q(v).
	\]

	We have
	\[
		R(Q^\beta,(1,\alpha)) = R(Q,\alpha) \oplus \bigoplus_{i \in Q_0} \Hom_k(E_i,V_i)
	\]
	and $G(1,\alpha) = k^\times \times G(\alpha)$. An element $(t,g)$ of the group $G(1,\alpha)$ acts on $(M,A) \in R(Q^\beta,(1,\alpha))$ by
	\[
		(t,g) \cdot (M,A) = ((g_{t(a)}M_ag_{s(a)}^{-1})_a,(g_iA_it^{-1})_i).
	\]
	The group $PG(1,\alpha)$ is isomorphic to $G(\alpha)$ and the induced group action on $R(Q^\beta,(1,\alpha))$ is given by $g \cdot (M,A) = ((g_{t(a)}M_a\smash{g_{s(a)}^{-1}})_a,(g_iA_i)_i)$.
	
	For the other framed quiver, we get
	\[
		R(Q_\beta,(\alpha,1)) = R(Q,\alpha) \oplus \bigoplus_{i \in Q_0} \Hom_k(V_i,E_i).
	\]
	The group $G(\alpha,1)$ is also $G(\alpha) \times k^\times$ and the induced action of $PG(\alpha,1) \cong G(\alpha)$ on $R(Q_\beta,(\alpha,1))$ is $g\cdot (M,B) = ((g_{t(a)}M_a\smash{g_{s(a)}^{-1}})_a,(B_i\smash{g_i^{-1}})_i)$.
	
	Consider the following stability parameters $c^\alpha$ for $Q^\beta$ and $c_\alpha$ for $Q_\beta$ given by
	\begin{align*}
		(c^\alpha)_0 &= \left|\alpha\right| & 
		(c_\alpha)_\infty &= -\left|\alpha\right| \\
		(c^\alpha)_i &= -1 &
		(c_\alpha)_i &= 1.
	\end{align*}
	In the above equations, $\left|\alpha\right|$ denotes the 1-norm of $\alpha$; note that $\alpha$ has non-negative entries, so $\left|\alpha\right| = \sum_{i \in Q_0} \alpha_i$.
	
	\begin{lem}[Engel--Reineke] \label{l:engel_reineke}
		Let $M \in R(Q,\alpha)$, let $A \in \bigoplus_{i \in Q_0} \Hom_k(E_i,V_i)$, and let $B \in \bigoplus_{i \in Q_0} \Hom_k(V_i,E_i)$.
		\begin{enumerate}
			\item The following are equivalent:
			\begin{enumerate}
				\item $(M,A)$ is $c^\alpha$-semi-stable.
				\item $(M,A)$ is $c^\alpha$-stable.
				\item No proper subrepresentation of $M$ contains $\im A := \smash{\bigoplus_{i \in Q_0}} \im A_i$.
				\item $\phi_{M,A}: \smash{P_\beta} \to M$ is surjective.
			\end{enumerate}
			\item The following are equivalent:
			\begin{enumerate}
				\item $(M,B)$ is $c_\alpha$-semi-stable.
				\item $(M,B)$ is $c_\alpha$-stable.
				\item Non non-zero subrepresentation of $M$ is contained in $\ker B := \smash{\bigoplus_{i \in Q_0}} \ker B_i$.
				\item $\psi_{M,B}: M \to \smash{I_\beta}$ is injective.
			\end{enumerate}
		\end{enumerate}
	\end{lem}
	
	\begin{proof}
		Part (1) of the lemma is shown in \cite{ER:09}. Part (2) can be proved analogously.
	\end{proof}

	We denote the (semi\nobreakdash-)stable locus with respect to $c^\alpha$ by $R(Q^\beta,(1,\alpha))^0$ and the (semi\nobreakdash-)\-stable locus with respect to $c_\alpha$ by $R(Q_\beta,(\alpha,1))^0$. So, concretely
	\begin{align*}
		R(Q^\beta,(1,\alpha))^0 &= \{(M,A) \mid \phi_{M,A}: P_\beta \to M \text{ is surjective}\} \\
		R(Q_\beta,(\alpha,1))^0 &= \{ (M,B) \mid \psi_{M,B}: M \to I_\beta \text{ is injective}\}.
	\end{align*}
	
	\begin{defn}
		The geometric quotients
		\begin{align*}
			\Hilb^{\beta,\alpha}(Q) &= R(Q^\beta,(1,\alpha))^0/G(\alpha) \\
			\Hilb_{\alpha,\beta}(Q) &= R(Q_\beta,(\alpha,1))^0/G(\alpha)
		\end{align*}
		are called \emph{non-commutative Hilbert schemes}.
	\end{defn}

	\begin{prop} \label{p:identification_NCHS_QGr}
		\leavevmode
		\begin{enumerate}
			\item The map $R(Q^\beta,(1,\alpha))^0 \to \Gr^\alpha(\smash{P_\beta})$ defined by $(M,A) \to \ker \phi_{M,A}$ is a well-defined algebraic map and descends to an isomorphism $\Phi: \Hilb^{\beta,\alpha}(Q) \to \Gr^\alpha(\smash{P_\beta})$.
			\item The map $R(Q_\beta,(\alpha,1))^0 \to \Gr_\alpha(\smash{I_\beta})$ defined by $(M,B) \to \im \psi_{M,B}$ is a well-defined algebraic map and descends to an isomorphism $\Psi: \Hilb_{\alpha,\beta}(Q) \to \Gr_\alpha(\smash{I_\beta})$.
		\end{enumerate}

	\end{prop}

	\begin{proof}
		By \cite[Subsect.\ 2.3]{CFR:12}, we know that $\Gr_\alpha(I_\beta)$ is isomorphic to 
		\[
		\Hom_Q^0(\alpha,I_\beta)/G(\alpha),
		\] 
		where $\Hom_Q(\alpha,I_\beta)$ is the set of all $(M,B) \in R(Q,\alpha) \times \prod_{i \in Q_0} \Hom(V_i,I_{\beta,i})$ such that $B_{t(a)}M_a = I_{\beta,a}B_{s(a)}$ for all $a \in Q_1$ and $\smash{\Hom_Q^0(\alpha,I_\beta)}$ is the subset where all $B_i$ are injective. The action of $G(\alpha)$ is given by $g \cdot (M,B) = (g\cdot M, (g_iB_i)_i)$. The map
		\[
			R(Q_\beta,(\alpha,1)) = R(Q,\alpha) \times \prod_{i \in Q_0} \Hom(V_i,E_i) \to \Hom_Q(\alpha,I_\beta)
		\]
		given by $(M,B) \mapsto (M,\psi_{M,B})$ is a $G(\alpha)$-equivariant isomorphism of varieties. The inverse image of $\smash{\Hom_Q^0(\alpha,I_\beta)}$ is precisely $R(Q_\beta,(\alpha,1))^0$. This proves (2). The proof of (1) is analogous.
	\end{proof}
	
	We will observe in the following that the isomorphisms $\Phi$ and $\Psi$ are equivariant with respect to the groups $G(\beta)$. We will define the actions of the group $G(\beta)$ on non-commutative Hilbert schemes and quiver Grassmannians in the following.
	
	To explain how the group $G(\beta)$ acts on non-commutative Hilbert schemes, we introduce two slightly different framing constructions. 	
	Consider the quivers $Q^\triangledown$ and $Q_\triangledown$ given as follows. Let $\smash{Q^\triangledown_0} = Q_0 \times \{0,1\}$ and $\smash{Q^\triangledown_1} = \{ (a,1): (s(a),1) \to (t(a),1) \} \sqcup \{ (i,01) : (i,0) \to (i,1) \}$. In words, we add a copy of every vertex and add an arrow from the copy to the original vertex. On the other hand, let $Q_{\triangledown,0} = Q_0 \times \{1,\infty\}$ and $Q_{\triangledown,1} = \{ (a,1): (s(a),1) \to (t(a),1) \} \sqcup \{ (i,1\infty) : (i,1) \to (i,\infty) \}$. That means, we again make a copy of each vertex, but this time add an arrow from the original vertex to its copy.
	
	Consider the dimension vector $(\beta,\alpha)$ for $Q^\triangledown$; formally $(\beta,\alpha)_{(i,0)} = \beta_i$ and $(\beta,\alpha)_{(i,1)} = \alpha_i$. Then
	\[
		R(Q^\triangledown,(\beta,\alpha)) = R(Q,\alpha) \oplus \bigoplus_{i \in Q_0} \Hom_k(E_i,V_i) = R(Q^\beta,(1,\alpha))
	\]
	but $G(\beta,\alpha) = G(\beta) \times G(\alpha)$; an element $(h,g) \in G(\beta) \times G(\alpha)$ acts on $(M,A) \in R(Q^\triangledown,(\beta,\alpha))$ by
	\[
		(h,g) \cdot (M,A) = ((g_{t(a)}M_ag_{s(a)}^{-1})_a,(g_iA_ih_i^{-1})) =: (g\cdot M, gAh^{-1}).
	\]
	For $Q_\triangledown$ and the dimension vector $(\alpha,\beta)$ defined by $(\alpha,\beta)_{(i,1)} = \alpha_i$ and $(\alpha,\beta)_{(i,\infty)} = \beta_i$, we have
	\[
		R(Q_\triangledown,(\alpha,\beta)) = R(Q,\alpha) \oplus \bigoplus_{i \in Q_0} \Hom_k(V_i,E_i) = R(Q_\beta,(\alpha,1))
	\]
	and $(g,h) \in G(\alpha,\beta) = G(\alpha) \times G(\beta)$ acts on $(M,B)$ by
	\[
		(g,h) \cdot (M,B) = ((g_{t(a)}M_ag_{s(a)}^{-1})_a,(h_iB_ig_i^{-1})) =: (g\cdot M, hBg^{-1})
	\]
	
	Via the identification $R(Q^\beta,(1,\alpha)) = R(Q^\triangledown,(\beta,\alpha))$ and $R(Q_\beta,(\alpha,1)) = R(Q_\triangledown,(\alpha,\beta))$ we transport the actions of $G(\beta,\alpha)$ and $G(\alpha,\beta)$ to $R(\smash{Q^\beta},(1,\alpha))$ and $R(Q_\beta,(\alpha,1))$, respectively. With respect to these group actions, the open subsets $R(\smash{Q^\beta},(1,\alpha))^0$ and $R(Q_\beta,(\alpha,1))^0$ are invariant. In order to obtain a group action on the quotients, we consider the following general lemma.

	\begin{lem} \label{l:group_action}
		Let $\Gamma$ be a reductive algebraic group and let $G$ be a closed normal subgroup of $\Gamma$. Let $X$ be a variety equipped with an action of $\Gamma$ such that a categorical $G$-quotient $\pi: X \to Y$ exists. Then there exists an action of $H := \Gamma/G$ on $Y$ such that $\pi$ is equivariant with respect to $\Gamma \to H$.
	\end{lem}
	
	\begin{proof}
		This is well-known. We give the argument for completeness and to fix notation.
		Let $\varpi: \Gamma \to H$. The morphism $\varpi \times \pi: \Gamma \times X \to H \times Y$ is a categorical $G \times G$-quotient. As the composition
		\[
			\Gamma \times X \xto{}{a_\Gamma} X \xto{}{\pi} Y
		\]
		of the action map $a_\Gamma$ with $\pi$ is $G \times G$-invariant, we obtain a unique morphism $a_{H}: H \times Y \to Y$ such that
		\[\begin{tikzcd}
			{\Gamma \times X} &[1em] X \\
			{H \times Y} & Y
			\arrow["{\varpi \times \pi}"', from=1-1, to=2-1]
			\arrow["{a_\Gamma}", from=1-1, to=1-2]
			\arrow["\pi"', from=1-2, to=2-2]
			\arrow["{a_{H}}", dashed, from=2-1, to=2-2]
		\end{tikzcd}\]
		is commutative. The map $a_{H}$ is an action of $H$. For closed points $h = \gamma G \in H$ and $x \in X$ we have $h\cdot \pi(x) = \pi(\gamma x)$.	
	\end{proof}
	
	We apply this to $\Gamma = PG(\beta,\alpha)$ (or $\Gamma = PG(\alpha,\beta)$), $G = G(\alpha)$, $H = PG(\beta)$, and $X = \smash{R(Q^\beta,(1,\alpha))^0}$ (or $X = \smash{R(Q_\beta,(\alpha,1))^0}$). \Cref{l:group_action} provides us with actions of $PG(\beta)$ on $\smash{\Hilb^{\beta,\alpha}}(Q)$ and $\smash{\Hilb_{\alpha,\beta}}(Q)$.
	
	Now to the action on quiver Grassmannians.
	For $h \in G(\beta)$, we regard $h_i \in \GL(E_i)$ as an automorphism of $P(i) \otimes E_i$. Therefore, $h$ gives rise to an automorphism $\sigma(h)$ of $P_\beta$. The map $\sigma: G(\beta) \to \Aut_A(P_\beta)$ is a morphism of algebraic groups. This provides us with an action of $G(\beta)$ on $\Gr^\alpha(P_\beta)$. If $h = (t\cdot \id_{E_i})$ for $t \in k^\times$, then $\sigma(h)$ is the multiplication with $t$ which leaves every subrepresentation invariant. Thus the $G(\beta)$-action on $\Gr^\alpha(P_\beta)$ descends to an action of $PG(\beta)$.
	In the same vein, there is a morphism of algebraic groups $G(\beta) \to \Aut_A(I_\beta)$; this yields an action of $PG(\beta)$ on $\Gr_\alpha(I_\beta)$. 
	
	\begin{lem} \label{l:equivariance}
		The isomorphisms $\Phi: \Hilb^{\beta,\alpha}(Q) \to \Gr^\alpha(\smash{P_\beta})$ and $\Psi: \Hilb_{\alpha,\beta}(Q) \to \Gr_\alpha(\smash{I_\beta})$ are $PG(\beta)$-equivariant.
	\end{lem}
	
	\begin{proof} 
		For $(M,A) \in R(Q^\triangledown,(\beta,\alpha))$ and $(h,g) \in G(\beta,\alpha)$, the homomorphism $\phi_{g\cdot M,gAh^{-1}}$ applied to $p \otimes w$ with $w \in E_{s(p)}$ is
		\begin{align*}
			\phi_{g\cdot M,gAh^{-1}}(p \otimes w) &= (g_{t(p)}M_pg_{s(p)}^{-1})(g_{s(p)}A_{s(p)}h_{s(p)}^{-1})(w) \\
				&= g_{t(p)}\phi_{M,Ah^{-1}}(w)
		\end{align*}
		which shows
		\[
			\ker \phi_{g\cdot M,gAh^{-1}} = \ker \phi_{M,Ah^{-1}}.
		\]
		This proves equivariance of $\Phi$.
		Similarly, for $(M,B) \in R(Q_\triangledown,(\alpha,\beta))$ and $(g,h) \in G(\alpha,\beta)$, the homomorphism $\psi_{g\cdot M,hBg^{-1}}$ maps $v_i \in V_i$ to
		\begin{align*}
			\psi_{g\cdot M,hBg^{-1}}(v) &= \sum_{\substack{q \in Q_*\\ s(q) = i}} q^* \otimes (h_{t(q)}B_{t(q)}g_{t(q)}^{-1})(g_{t(q)}M_qg_{s(q)}^{-1})(v) \\
				&= \psi_{M,hB}(g_i^{-1}v)
		\end{align*}
		and therefore
		\[
			\im \psi_{g\cdot M,hBg^{-1}} = \im \psi_{M,hB}.
		\]
		Equivariance of $\Psi$ is proved.
	\end{proof}
	
	We describe the induced group actions on the quiver Grassmannians for our two running examples. In these cases, $\beta$ is always $\smash{\alpha^\pm}$.
	
	\begin{ex} \label{e:ex1.3}
		Recall Examples \ref{e:ex1.1} and \ref{e:ex1.2}. We describe the group actions on $\Gr^\alpha(P^+)$ and $\Gr_\alpha(I^-)$. 
		
		\noindent
		For $\Gr^\alpha(P^+) \cong \Gr^n(k^m)$, we get $\smash{\alpha^+} = (1,\ldots,1,0)$ and we obtain
		\[
			R(Q^{\alpha^+},(1,\alpha)) = k^m \oplus (k^n)^m = R(Q^\triangledown,(\alpha^+,\alpha)).
		\]
		The open subset $R(Q^{\alpha^+},(1,\alpha))^0$ consists of all $((z_1,\ldots,z_m),(v_1,\ldots,v_m))$ such that all $z_i \neq 0$ and the matrix $(v_1\mid \ldots \mid v_m)$ has full rank.
		We get an induced action of $PG(\smash{\alpha^+}) = (k^\times)^m/\Delta$ on $\Gr(P^+)$. A group element $h$, represented by $(h_1,\ldots,h_m) \in (k^\times)^m$, acts on an element $Q \in \Gr^n(k^m)$, represented by a matrix $A = (v_1\mid \ldots \mid v_m) \in M_{n \times m}(k)$ of full rank as $Q\cdot h^{-1}$, which is represented by the matrix
		\[
			A\cdot \begin{pmatrix} h_1^{-1} & & \\ & \ddots & \\ & & h_m^{-1} \end{pmatrix}.
		\]
		
		\noindent
		Now to $\Gr_\alpha(I^-) \cong (\P^1)^m$. We have $\smash{\alpha^-} = (0,\ldots,0,n)$ and
		\[
			R(Q_{\alpha^-},(\alpha,1)) = (k^2)^m \oplus M_{n \times n}(k)  = R(Q_\triangledown,(\alpha,\alpha^-)).
		\]
		The set $R(Q_{\alpha^-},(\alpha,1))^0$ consists of all $((v_1,\ldots,v_m),A)$ where all $v_i \neq 0$ and $A$ is invertible.
		The induced action is as follows. An element $h$ of group $PG(\smash{\alpha^-}) = \PGL_n(k)$ acts on $(x_1,\ldots,x_m) \in (\P^{n-1})^m$ by $(hx_1,\ldots,hx_m)$.
	\end{ex}
	
	\begin{ex} \label{e:ex2.3}
		Recall Examples \ref{e:ex2.1} and \ref{e:ex2.2}. We describe the group actions on $\Gr^\alpha(P)$ and $\Gr_\alpha(I)$. Both are isomorphic to Schubert varieties of $\GL(V)/H$ and $\alpha^+ = \alpha^- = \alpha$, so the group which acts is $G(\alpha) = \prod_{i=1}^n \GL(V_i)$. On both varieties, the action of $PG(\alpha^\pm) = PG(\alpha)$ comes from the action of $G(\alpha) = \prod_{i=1}^n \GL(V_i) \sub \GL(V)$, which acts on $\GL(V)/H$ by left multiplication.

	\end{ex}

	\section{Descent of line bundles} \label{s:descent}

	Now let $\theta \in \Z^{Q_0}$ be a stability parameter for $Q$ such that $\theta(\alpha) = 0$ and let $\beta = \smash{\alpha^+}$.
	Pick a natural number $N > 0$. We define a stability parameter $\theta^+$ for $Q^\triangledown$ by
	\begin{align*}
		\theta^+_{(i,0)} &= \begin{cases} N & \text{if } \theta_i \geq 0 \\ 0 & \text{otherwise} \end{cases} &
		\theta^+_{(i,1)} &= \begin{cases} -N+\theta_i & \text{if } \theta_i \geq 0 \\ \theta_i & \text{otherwise.} \end{cases}
	\end{align*}
	We see that $\theta^+(\alpha^+,\alpha) = 0$. 
	We also define the stability condition $\theta^-$ for $Q_\triangledown$ by
	\begin{align*}
		\theta^-_{(i,1)} &= \begin{cases} N+\theta_i & \text{if } \theta_i \leq 0 \\ \theta_i & \text{otherwise.} \end{cases} &
		\theta^-_{(i,\infty)} &= \begin{cases} -N & \text{if } \theta_i \leq 0 \\ 0 & \text{otherwise} \end{cases} 
	\end{align*}	
	We analyze (semi\nobreakdash-)stability with respect to $\theta^\pm$.
	
	\begin{prop} \label{p:theta+-stability}
		Let $N \gg 0$. 
		\begin{enumerate}
			\item An element $(M,A) \in R(Q^\triangledown,(\alpha^+,\alpha))$ is $\theta^+$-(semi\nobreakdash-)stable if and only if $M$ is $\theta$-(semi\nobreakdash-)stable and $A_i$ is invertible for every $i \in Q_0^+$.
			\item An element $(M,B) \in R(Q_\triangledown,(\alpha,\alpha^-))$ is $\theta^-$-(semi\nobreakdash-)stable if and only if $M$ is $\theta$-(semi\nobreakdash-)stable and $B_i$ is invertible for every $i \in Q_0^-$.
		\end{enumerate}
	\end{prop}

	\begin{proof}
		We prove (1); the proof of (2) is analogous.
		Let $(\beta',\alpha')$ be another dimension vector of $Q^\triangledown$. We compute
		\[
			\theta^+(\beta',\alpha') = N(\left|\beta'\right| - \left|\smash{{\alpha'}^+}\right|) + \theta(\alpha').
		\]
		If we choose $N$ sufficiently big then $\theta^+(\beta',\alpha') > 0$ if and only if either of the following hold:
		\begin{itemize}
			\item $\left|\beta'\right| > \left|\smash{{\alpha'}^+}\right|$
			\item $\left|\beta'\right| = \left|\smash{{\alpha'}^+}\right|$ and $\theta(\alpha') > 0$.
		\end{itemize}
		Moreover $\theta^+(\beta',\alpha') = 0$ if and only if $\left|\beta'\right| = \left|\smash{{\alpha'}^+}\right|$ and $\theta(\alpha') = 0$.
		
		\noindent
		Let us assume that $(M,A)$ is $\theta^+$-semi-stable. If $A_i$ were not injective for a vertex $i \in Q_\theta$ then $(M,A)$ would have a subrepresentation of dimension vector $(e_i,0)$; here $e_i$ is the simple root at $i$. But $\theta^+(e_i,0) = N > 0$. Now to the $\theta$-semi-stability of $M$. Let $\alpha'$ be a dimension vector such that $\theta(\alpha') > 0$ and suppose that $M$ had a subrepresentation $(U_i)_{i \in Q_0}$ of dimension vector $\alpha'$. Then 
		\[
			((A_i^{-1}(U_i))_{i \in Q_\theta},(U_i)_{i \in Q_0})
		\]
		would be a subrepresentation of $(M,A)$ of dimension vector $(\smash{{\alpha'}^+},\alpha')$. But $\theta^+(\smash{{\alpha'}^+},\alpha') = \theta(\alpha') > 0$. The same argument shows that for a $\theta^+$-stable representation $(M,A)$, it is necessary that $M$ is $\theta$-stable.
		
		\noindent
		Conversely, assume that $M$ is $\theta$-semi-stable and all $A_i$ are invertible. Let $(\beta',\alpha')$ be a dimension vector with $\theta^+(\beta',\alpha') > 0$. We distinguish the above two cases:
		\begin{itemize}
			\item If $\left|\beta'\right| > \left|\smash{{\alpha'}^+}\right|$ then there must be a vertex $i \in Q_0^+$ such that $\beta'_i > \alpha'_i$. If $(M,A)$ had a subrepresentation of dimension vector $(\beta',\alpha')$, then $A_i$ would map a $\beta'_i$-dimensional subspace into an $\alpha'_i$-dimensional subspace. This contradicts injectivity of $A_i$.
			\item If $\left|\beta'\right| = \left|\smash{{\alpha'}^+}\right|$ and $\theta(\alpha') > 0$ then by $\theta$-semi-stability, $M$ has no subrepresentation of dimension vector $\alpha'$. Therefore $(M,A)$ cannot have a subrepresentation of dimension vector $(\beta',\alpha')$.
		\end{itemize}
		This shows that $(M,A)$ is $\theta^+$-semi-stable. 
		
		\noindent
		Finally, let us assume that $M$ is $\theta$-stable and all $A_i$ are invertible. If $(\beta',\alpha')$ is a dimension vector such that $\theta^+(\beta',\alpha') = 0$, then $\left|\beta'\right| = \left|\smash{{\alpha'}^+}\right|$ and $\theta(\alpha') = 0$ must hold. Assume that $(M,A)$ had a subrepresentation of dimension vector $(\beta',\alpha')$. Then $M$ has a subrepresentation of dimension vector $\alpha'$, so by $\theta$-stability, $\alpha' \in \{0,\alpha\}$. If $\alpha' = 0$ then $\beta' = 0$, while if $\alpha'=\alpha$ then $\beta' = \alpha^+$. This proves $\theta^+$-stability of $(M,A)$.
	\end{proof}

	We consider the quotient maps
	\begin{align*}
		\pi^+: R(Q^{\alpha^+},(1,\alpha))^0 &\to \Hilb^{\alpha^+,\alpha}(Q) & 
		\pi^-: R(Q_{\alpha^-},(\alpha,1))^0 &\to \Hilb_{\alpha,\alpha^-}(Q)
	\end{align*}
	by $G(\alpha)$ and the actions of $PG(\alpha^\pm)$ on the non-commutative Hilbert schemes provided by \Cref{l:group_action}.
	
	Before we proceed we recall some general facts on descent of equivariant vector bundles. Let $X$ be a quasi-projective variety equipped with an action of a reductive algebraic group $G$. We suppose that a good $G$-quotient $\pi: X \to Y$ exists. Let $E$ be a $G$-equivariant vector bundle $E$ on $X$. A descent of $E$ to $Y$ is a vector bundle $F$ on $Y$ together with an isomorphism $\pi^*F \to E$ of $G$-equivariant vector bundles. Note that, if it exists, a descent is unique. To see this, we follow an argument laid out by Knop in an answer on mathoverflow: As $\pi$ is a good quotient, the natural map $\OO_Y \to \pi_*\pi^*\OO_Y = \pi_*\OO_X$ factors through an isomorphism 
	$\smash{\OO_Y \xto{}{\cong} (\pi_*\OO_X)^G}$. 
	By local triviality of $F$, we get $\smash{F \xto{}{\cong} (\pi_*\pi^*F)^G}$. As the isomorphism $\pi^*F \to E$ is $G$-equivariant, it yields an isomorphism $(\pi_*\pi^*F)^G \to (\pi_*E)^G$. The composition gives an isomorphism 
	\[
		F \xto{}{\cong} (\pi_*E)^G.
	\]

	\begin{prop} \label{p:descent}
		Let $\Gamma$ be a reductive algebraic group, $G \sub \Gamma$ a closed normal subgroup, $H := \Gamma/G$, and $X$ a variety with a $\Gamma$-action such that a good $G$-quotient $\pi: X \to Y$ exists. Let $E$ be a $\Gamma$-equivariant vector bundle on $X$ such that for every closed point $x \in X$ which has a closed $G$-orbit, the stabilizer $G_x$ acts trivially on the fiber $E_x$. Then $E$ descends to an $H$-equivariant vector bundle on $Y$.
	\end{prop}

	\begin{proof}
		The case where $G = \Gamma$ is Kempf's descent criterion, see \cite[Thm.\ 2.3]{DN:89}. Let $l: a_\Gamma^*E \to \pr_2^*E$ be the $\Gamma$-linearization of $E$; here $a_\Gamma: \Gamma \times X \to X$ is the action and $\pr_2: \Gamma \times X \to X$ is the projection. We obtain a $G$-linearization $l|_G$ which fulfills Kempf's criterion. This yields the existence of a vector bundle $F$ on $Y$ with an isomorphism $\eta: \pi^*F \to E$ of $G$-equivariant vector bundles. Consider the good $G \times G$-quotient $\varpi \times \pi: \Gamma \times X \to H \times Y$ and the isomorphism $\lambda$ such that the following diagram is commutative:
		\[\begin{tikzcd}
			{a_\Gamma^*\pi^*F} &[-1em] {(\varpi \times \pi)^*a_H^*F} & {(\varpi \times \pi)^*\pr_2^*F} &[-1em] {\pr_2^*\pi^*F} \\
			{a_\Gamma^*E} &&& {\pr_2^*E}
			\arrow["{a_\Gamma^*\eta}", from=1-1, to=2-1]
			\arrow[equals, from=1-1, to=1-2]
			\arrow["\lambda", from=1-2, to=1-3]
			\arrow[equals, from=1-3, to=1-4]
			\arrow["{\pr_2^*\eta}", from=1-4, to=2-4]
			\arrow["l", from=2-1, to=2-4]
		\end{tikzcd}\]
		The isomorphism $\lambda$ is $G \times G$-equivariant. To see this, we can apply $G$-equivariance of $\eta$ and the cocycle condition which reduces the claim to showing that the diagram of isomorphisms of vector bundles on $G \times \Gamma \times X$
		\[\begin{tikzcd}
			{(\gamma x)^*\pi^*F} &[-1em] {(\gamma gx)^*\pi^*F} &[1em] {(\gamma gx)^*E} &[-1em] {(\gamma g,x)^*a_\Gamma^*E} \\
			{(\gamma x)^*E} & {(\gamma,x)^*a_\Gamma^*E} & {(\gamma,x)^*\pr_2^*E} & {(\gamma g,x)^*\pr_2^*E}
			\arrow["{(\gamma x)^*\eta}", from=1-1, to=2-1]
			\arrow[equals, from=1-1, to=1-2]
			\arrow["{(\gamma gx)^*\eta}", from=1-2, to=1-3]
			\arrow[equals, from=1-3, to=1-4]
			\arrow[equals, from=2-1, to=2-2]
			\arrow["{(\gamma,x)^*l}", from=2-2, to=2-3]
			\arrow[equals, from=2-3, to=2-4]
			\arrow["{(\gamma g,x)^*l}", from=1-4, to=2-4]
		\end{tikzcd}\]
		commutes. Here, we have used the following short-hand notations: for the morphism $G \times \Gamma \times X \to \Gamma \times X$ which sends a closed point $(g,\gamma,x)$ to $(\gamma g,x)$ we write $(\gamma g,x)$. The other maps are defined analogously. For $g \in G$ and $\gamma \in \Gamma$ define $g^\gamma := \gamma g\gamma^{-1} \in G$. Then $\gamma g = g^\gamma\gamma$. By the cocycle condition
		\begin{align*}
			(\gamma g,x)^*l = (g^\gamma,\gamma,x)^*(m \times \id)^*l = (g^\gamma,\gamma,x)^*(\pr_{23}^*l \circ (\id \times a_\Gamma)^* l) = (\gamma,x)^*l \circ (g^\gamma,\gamma x)^*l.
		\end{align*}
		Thus, using $G$-equivariance of $\eta$
		\begin{align*}
			(\gamma g,x)^*l \circ (\gamma gx)^*\eta &= (\gamma,x)^*l \circ (g^\gamma,\gamma x)^*l \circ (g^\gamma\gamma x)^*\eta = (\gamma,x)^*l \circ (g^\gamma,\gamma x)^*(l|_G \circ a_G^*\eta) \\
			&= (\gamma,x)^*l \circ (g^\gamma,\gamma x)^*\pr_2^*\eta = (\gamma,x)^*l \circ (\gamma x)^*\eta.
		\end{align*}
		This shows equivariance of $\lambda$. We consider $l_F: a_H^*F \xto{}{\cong} \pr_2^*F$ given by
		\[
			a_H^*F = ((\varpi \times \pi)_*(\varpi \times \pi)^*a_H^*F)^{G \times G} \xto{}{((\varpi \times \pi)_*\lambda)^G} ((\varpi \times \pi)_*(\varpi \times \pi)^*\pr_2^*F)^{G \times G} = \pr_2^*F
		\]
		This can be shown to be an $H$-linearization of $F$, i.e.\ satisfies the cocycle condition, and it makes $\eta: \pi^*F \to E$ a $\Gamma$-equivariant isomorphism.
	\end{proof}

	We apply \Cref{p:descent} to the bundles $L(\theta^+)$ on $R(Q^{\alpha^+},(1,\alpha))^0$ and $L(\theta^-)$ on $R(Q_{\alpha^-},(\alpha,1))^0$.

	\begin{cor}
		\leavevmode
		\begin{enumerate}
			\item There exists a unique $\smash{PG({\alpha^+})}$-equivariant line bundle $\mathcal{L}(\theta^+)$ on $\smash{\Hilb^{\alpha^+,\alpha}}(Q)$ such that $(\pi^+)^*\mathcal{L}(\theta^+) \cong L(\theta^+)$ as $\smash{PG(\alpha^+},\alpha)$-equivariant line bundles. 
			\item There exists a unique $\smash{PG({\alpha^-})}$-equivariant line bundle $\mathcal{L}(\theta^-)$ on $\smash{\Hilb_{\alpha,\alpha^-}}(Q)$ such that $(\pi^-)^*\mathcal{L}(\theta^-) \cong L(\theta^-)$ as $\smash{PG(\alpha,\alpha^-)}$-equivariant line bundles. 
		\end{enumerate}
	\end{cor}
	
	Next, we are concerned with ampleness of the descended line bundles on geometric quotients. To this end, recall the definition of (semi\nobreakdash-)stability on a quasi-projective variety with respect to a not necessarily ample line bundle:
		
	\begin{defn}
		Let $X$ be a quasi-projective variety equipped with an action of a reductive algebraic group $G$ and let $L$ be a $G$-equivariant line bundle. Let $x \in X$.
		\begin{enumerate}
			\item Call $x$ \emph{semi-stable} with respect to $L$ if there exists $n > 0$ and $s \in H^0(X,L^{\otimes n})^G$ such that $s(x) \neq 0$ and $X_s$ is affine.
			\item Call $x$ \emph{stable} with respect to $L$ if there exists $n > 0$ and $s \in H^0(X,L^{\otimes n})^G$ such that $s(x) \neq 0$, $X_s$ is affine, and the action of $G$ on $X_s$ is closed.
		\end{enumerate}
	\end{defn}
	
	In the above definition, $X_s = \{x \in X \mid s(x) \neq 0 \}$ and an action is called closed if all orbits are closed.

	\begin{lem} \label{l:general_ample}
		Let $X$ be a quasi-projective variety with an action of a reductive algebraic group $G$. Let $L$ be a $G$-equivariant line bundle and let $Y := X^{L\hypsst}/\!\!/G$. Let $L'$ be another $G$-equivariant line bundle such that
		\[
			X^{L'\hypsst} = X^{L\hypsst}.
		\]
		If $L'$ descends to a line bundle $\mathcal{L}'$ on $Y$ then $\mathcal{L}'$ is ample.
	\end{lem}

	\begin{proof}
		Abbreviate $X^0 = X^{L\hypsst}$. Let $\pi: X^0 \to Y$ be the quotient map. As $\pi$ is a good $G$-quotient, we have an isomorphism $\smash{\mathcal{L}' \xto{}{\cong} (\pi_*L')^G}$. An application of the global sections functor gives an isomorphism
		\[
			H^0(Y,\mathcal{L}') \xto{}{\cong} H^0(X^0,L')^G.
		\]
		The isomorphism maps a section $s$ to $\pi^*s$.
		The same holds for all powers of $\mathcal{L}'$. Let $y \in Y$ and let $x \in X^0$ such that $y = \pi(x)$. Then $x \in X^{L'\hypsst}$, therefore there exists $n > 0$ and $t \in H^0(X,{L'}^{\otimes n})$ such that $t(x) \neq 0$ and $X_t$ is affine. Let $s \in H^0(Y,{\mathcal{L}'}^{\otimes n})$ such that $\pi^*s = t|_{X^0}$. Then obviously $s(y) = s(\pi(x)) = t(x) \neq 0$. We also get
		\[
			\pi^{-1}(Y_s) = (X_0)_{t|_{X^0}} = X_t;
		\]
		the latter equality uses $X^{L'\hypsst} \sub X^0$. The restriction $\pi|_{X_t}: X_t \to Y_s$ is therefore also a good $G$-quotient. As $X_t$ is affine, say $X_t = \Spec A$, we get $Y_s = \Spec A^G$, also affine. We have proved ampleness of $\mathcal{L}'$.
	\end{proof}

	We would like to show ampleness of the line bundles $\mathcal{L}(\theta^\pm)$. We argue for $\mathcal{L}(\theta^+)$, the argument for $\mathcal{L}(\theta^-)$ is analogous. To apply \Cref{l:general_ample}, we need to look at the line bundle $L(\theta^+)$ as a $G(\alpha)$-equivariant line bundle. As a bundle, it is trivial, and the equivariant structure is given by the character
	\[
		\chi_{\theta^+}|_{G(\alpha)}: G(\alpha) \to k^\times.
	\]
	To determine the stability parameter of $Q^{\alpha^+}$ which corresponds to this character, we need to pre-compose it with $G(1,\alpha) \to PG(1,\alpha) \cong G(\alpha)$. This morphism is the horizontal arrow in the diagram
	\[\begin{tikzcd}
		{(t,g)} & {(t^{-1}g_i)_i} \\[-2em]
		{G(1,\alpha)} & {G(\alpha)} \\
		{PG(1,\alpha)}
		\arrow[from=2-1, to=2-2]
		\arrow[from=2-1, to=3-1]
		\arrow["{\cong}"', from=3-1, to=2-2]
		\arrow[maps to, from=1-1, to=1-2]
	\end{tikzcd}\]
	The resulting character $\chi: G(1,\alpha) \to PG(1,\alpha) \xto{}{\cong} G(\alpha) \xto{}{\chi_{\theta^+}} k^\times$ sends $(t,g)$ to
	\[
		\chi(t,g) = \chi_{\theta^+}((t^{-1}g_i)_i) = \prod_{i \in Q_0^+} \det(t^{-1}g_i)^{N-\theta_i} \prod_{i \notin Q_0^+} \det(t^{-1}g_i)^{-\theta_i} = t^{-N\left|\smash{\alpha^+}\right|} \cdot \chi_{\theta^+}(1,g)
	\]	
	This shows us that $\chi$ is the character $\chi_{\eta^+}$, where $\eta^+$ is the stability parameter for $Q^{\alpha^+}$ given by
	\begin{align*}
		\eta^+_0 &= N\left|\smash{\alpha^+}\right| &
		\eta^+_i &= \begin{cases} -N + \theta_i & \text{if } i \in Q_0^+ \\ \theta_i & \text{otherwise.} \end{cases}
	\end{align*}
	Similarly, the character $G(\alpha,1) \to PG(\alpha,1) \xto{}{\cong} G(\alpha) \xto{}{\chi_{\theta^-}} k^\times$ equals $\chi_{\eta^-}$, where	
	\begin{align*}
		\eta^-_\infty &= -N\left|\smash{\alpha^-}\right| &
		\eta^-_i &= \begin{cases} N + \theta_i & \text{if } i \in Q_0^- \\ \theta_i & \text{otherwise.} \end{cases}
	\end{align*}
	
	We compare (semi\nobreakdash-)stability with respect to $\eta^{\pm}$ with (semi\nobreakdash-)stability with respect to $c^\alpha$ and $c_\alpha$, respectively.
	
	\begin{lem} \label{l:framed_stability}
		Let $N \gg 0$. The following hold:
		\begin{enumerate}
			\item $R(Q^{\alpha^+},(1,\alpha))^{\eta^+\hypsst} = R(Q^{\alpha^+},(1,\alpha))^{\eta^+\hypst} = R(Q^{\alpha^+},(1,\alpha))^0$
			\item $R(Q_{\alpha^-},(\alpha,1))^{\eta^-\hypsst} = R(Q_{\alpha^-},(\alpha,1))^{\eta^-\hypst} = R(Q_{\alpha^-},(\alpha,1))^0$
		\end{enumerate}
	\end{lem}
	
	\begin{proof}
		We prove the first assertion. The second is proved analogously.
		We analyze the values of sub-dimension vectors of $(1,\alpha)$. Let $\alpha'$ be a dimension vector for $Q$ such that $\alpha'_i \leq \alpha_i$ for every $i \in Q_0$. Assume that $\alpha' \neq \alpha$ and consider the dimension vector $(1,\alpha')$:
		\[
			\eta^+(1,\alpha') = N\cdot(\left|\smash{\alpha^+}\right|-\left|\smash{{\alpha'}^+}\right|) + \theta(\alpha')
		\]
		If $\left|\smash{\alpha^+}\right| > \left|\smash{{\alpha'}^+}\right|$ then $\eta^+(1,\alpha') > 0$ as $N$ is big. If $\left|\smash{\alpha^+}\right| = \left|\smash{{\alpha'}^+}\right|$, then there must exist $i \in Q_0 \setminus Q_0^+$ such that $\alpha'_i < \alpha_i$. This shows
		\[
			\theta(\alpha') = \underbrace{\theta({\alpha'}^+)}_{=\theta(\alpha^+)} + \underbrace{\theta(\alpha'-{\alpha'}^+)}_{> \theta(\alpha-\alpha^+)} > \theta(\alpha) = 0
		\]
		So dimension vectors of the form $(1,\alpha')$ with $\alpha' \neq \alpha$ always have a strictly positive value under $\eta^+$.
		Now to dimension vectors of the form $(0,\alpha')$ with $\alpha' \neq 0$. We compute
		\[
			\eta^+(0,\alpha') = -N\left|\smash{{\alpha'}^+}\right| + \theta(\alpha')
		\]
		If $\left|\smash{{\alpha'}^+}\right| > 0$ then $\eta^+(0,\alpha') < 0$ because $N \gg 0$. If $\left|\smash{{\alpha'}^+}\right| = 0$, then we obtain
		\[
		\theta(\alpha') = \theta(\alpha'-{\alpha'}^+) < 0
		\]
		as $\alpha' \neq 0$.
		We have asserted that $\eta^+(0,\alpha') < 0$ for every sub-dimension vector $\alpha' \neq 0$ of $\alpha$. This shows that a representation $(M,A)$ of dimension vector $(1,\alpha)$ is $\eta^+$-semi-stable if and only if it is $\eta^+$-stable if and only if there is no proper subrepresentation of $M$ which contains the image of $A$. By \Cref{l:engel_reineke}, this agrees with (semi\nobreakdash-)stability with respect to $c^\alpha$. 
	\end{proof}
	
	We conclude with Lemmas \ref{l:general_ample} and \ref{l:framed_stability}:	
	
	\begin{prop} \label{p:ample}
		The line bundles $\mathcal{L}(\theta^+)$ on $\smash{\Hilb^{\alpha^+,\alpha}}(Q)$ and $\mathcal{L}(\theta^-)$ on $\smash{\Hilb_{\alpha,\alpha^-}}(Q)$ are ample.
	\end{prop}

	In the last part of the section, we should like to compare (semi\nobreakdash-)stability with respect to $\mathcal{L}(\theta^\pm)$ with (semi\nobreakdash-)stability with respect to $L(\theta^\pm)$ and relate the quotients. This can again be treated in greater generality.
	
	\begin{lem} \label{l:general_stability_descent}
		Let $X$ be a quasi-projective variety equipped with an action of a reductive group $\Gamma$. Let $G$ be a closed normal subgroup of $\Gamma$ and let $H := \Gamma/G$. Suppose that there exists a good quotient $\pi: X \to Y$ by $G$. Let $\mathcal{L}$ be an $H$-equivariant line bundle on $Y$ and let $L := \pi^*\mathcal{L}$ considered as a $\Gamma$-equivariant line bundle. Then:
		\begin{enumerate}
			\item We have $\pi(X^{L\hypsst}) = Y^{\mathcal{L}\hypsst}$ and $\pi$ induces an isomorphism
			\[
				X^{L\hypsst}/\!\!/\Gamma \xto{}{\cong} Y^{\mathcal{L}\hypsst}/\!\!/H.
			\]
			\item If $\pi$ is a geometric quotient by $G$ with finite stabilizers then $\pi(X^{L\hypst}) = Y^{\mathcal{L}\hypst}$ and the isomorphism from above restricts to an isomorphism
			\[
				X^{L\hypst}/\Gamma \xto{}{\cong} Y^{\mathcal{L}\hypst}/H.
			\]
		\end{enumerate}
	\end{lem}
	
	\begin{proof}
		We abbreviate in the proof $X^{\sst} = X^{L\hypsst}$ and $Y^{\sst} = Y^{\mathcal{L}\hypsst}$; we use the analogous short-hand notation for the stable loci. 
				
		(1) Like in the proof of \Cref{l:general_ample}, the fact that $\pi$ is a good $G$-quotient implies that pulling back sections along $\pi$ yields an isomorphism
		\[
			H^0(Y,\mathcal{L}^{\otimes n}) \xto{}{\cong} H^0(X,L^{\otimes n})^G.
		\]
		A section $s \in H^0(Y,\mathcal{L}^{\otimes n})$ is $H$-invariant if and only if $\pi^*s$ is $\Gamma$-invariant. Moreover, $X_{\pi^*s} = \pi^{-1}(Y_s)$, a $G$-invariant open subset of $X$. Therefore, the restriction $X_{\pi^*s} \to Y_s$ of $\pi$ is also a good $G$-quotient which shows that $Y_s$ is affine if $X_{\pi^*s}$ is.
		
		\noindent
		From the isomorphism $\smash{H^0(Y,\mathcal{L}^{\otimes n}) \xto{}{\cong} H^0(X,L^{\otimes n})^G}$, we obtain an isomorphism of $\Z_{\geq 0}$-graded $k$-algebras
		\[ 
			\bigoplus_{n \geq 0} H^0(Y,\mathcal{L}^{\otimes n})^H \xto{}{\cong} \bigoplus_{n \geq 0} H^0(X,L^{\otimes n})^\Gamma
		\]
		and this shows that $\pi$ induces an isomorphism $X^{\sst}/\!\!/\Gamma \to Y^{\sst}/\!\!/H$.

		(2) We first show that for $x \in X$, if the stabilizer $\Gamma_x$ is finite then $\smash{H_{\pi(x)}}$ is finite too. 
		For $h = \gamma G \in \smash{H_{\pi(x)}}$, we have $\pi(x) = h\pi(x) = \pi(\gamma x)$ and as $\pi$ is a geometric quotient, it follows that $x = g\gamma x$ for some $g \in G$. But $g\gamma G = \gamma G = h$. Therefore, the map $\Gamma_x \to \smash{H_{\pi(x)}}$ defined by $\gamma \mapsto \gamma G$ is surjective.
		
		\noindent
		Now to the closedness of the actions. Let $s \in H^0(Y,\mathcal{L}^{\otimes n})^H$ and consider $t := \pi^*s \in H^0(X,L^{\otimes n})^{\Gamma}$. 
		Assume that the $\Gamma$-action on $X_t$ is closed. Let $y \in Y_s$ and $y' \in \smash{\bar{H\cdot y}}$, the closure in $Y_s$. Let $x' \in X$ such that $\pi(x') = y'$. Then $x' \in X_t$. Let $V \sub X_t$ be an open neighborhood of $x'$. Then $\pi(V) = \pi(G\cdot V) \sub Y_s$ is an open neighborhood of $\pi(x') = y'$, because a good quotient maps open $G$-invariant sets to open sets. As $y' \in \smash{\bar{H\cdot \pi(x)}}$, we find $h = \gamma G \in H$ such that $h\pi(x) = \pi(\gamma x)$ lies in $\pi(V)$. Take $x'' \in V$ such that $\pi(x'') = \pi(\gamma x)$. Then there exists $g \in G$ such that $x'' = g\gamma x$. This shows $V \cap \Gamma\cdot x \neq \emptyset$. As $V$ was an arbitrary open neighborhood of $x'$, we see that $x' \in \smash{\bar{\Gamma\cdot x}} = \Gamma\cdot x$. We conclude $y' = \pi(x') \in H\cdot \pi(x)$.
		
		\noindent
		To the claimed isomorphism of the stable quotients. Let $\pi_1: X^{\sst} \to X^{\sst}/\!\!/\Gamma$ and $\pi_2: Y^{\sst} \to Y^{\sst}/\!\!/H$ be the quotient maps and let $\phi: X^{\sst}/\!\!/\Gamma \smash{\xto{}{\cong}} Y^{\sst}/\!\!/H$ be the isomorphism from (1). We get
		\[
			\phi(X^{\st}/\Gamma) = \phi(\pi_1(X^{\st})) = \pi_2(\pi(X^{\st})) = \pi_2(Y^{\st}) = Y^{\st}/H
		\]
		and therefore $\phi$ restricts to an isomorphism as claimed.
	\end{proof}

	We apply \Cref{l:general_stability_descent} to our non-commutative Hilbert schemes. In this case, we have $X = R(Q^\triangledown,(\alpha^+,\alpha))^0$ or $X = R(Q_\triangledown,(\alpha,\alpha^-))^0$, $\Gamma = PG(\alpha^+,\alpha)$ or $\Gamma = PG(\alpha,\alpha^-)$, $G = G(\alpha)$, $H = PG(\alpha^\pm)$ and the quotient $\pi: X \to Y$ is either $\pi^+: \smash{R(Q^{\alpha^+},(1,\alpha))^0} \to \smash{\Hilb^{\alpha^+,\alpha}(Q)}$ or $\pi^-: R(Q_{\alpha^-},(\alpha,1))^0 \to \Hilb_{\alpha,\alpha^+}(Q)$. We obtain:
		
	\begin{prop} \label{p:quotients_Hilbert_schemes}
		\leavevmode
		\begin{enumerate}
			\item We have $\pi^+(R(Q^\triangledown,(\alpha^+,\alpha))^{\theta^+\hypsst}) = \Hilb^{\alpha^+,\alpha}(Q)^{\mathcal{L}(\theta^+)\hypsst}$ and the same holds for the stable loci. We obtain isomorphisms
			\[\begin{tikzcd}
				{M^{\theta^+\hypsst}(Q^\triangledown,(\alpha^+,\alpha))} & {\Hilb^{\alpha^+,\alpha}(Q)^{\mathcal{L}(\theta^+)\hypsst}/\!\!/PG(\alpha^+)} \\[-2em]
				{\text{\rotatebox{90}{$\sub$}}} & {\text{\rotatebox{90}{$\sub$}}} \\[-2em]
				{M^{\theta^+\hypst}(Q^\triangledown,(\alpha^+,\alpha))} & {\Hilb^{\alpha^+,\alpha}(Q)^{\mathcal{L}(\theta^+)\hypst}/PG(\alpha^+)}
				\arrow["\cong", from=1-1, to=1-2]
				\arrow["\cong", from=3-1, to=3-2]
			\end{tikzcd}\]
			
			\item We have $\pi^-(R(Q_\triangledown,(\alpha,\alpha^-))^{\theta^-\hypsst}) = \Hilb_{\alpha,\alpha^-}(Q)^{\mathcal{L}(\theta^-)\hypsst}$ and the same holds for the stable loci. We obtain isomorphisms
			\[\begin{tikzcd}
				{M^{\theta^-\hypsst}(Q_\triangledown,(\alpha,\alpha^-))} & {\Hilb_{\alpha,\alpha^-}(Q)^{\mathcal{L}(\theta^-)\hypsst}/\!\!/PG(\alpha^-)} \\[-2em]
				{\text{\rotatebox{90}{$\sub$}}} & {\text{\rotatebox{90}{$\sub$}}} \\[-2em]
				{M^{\theta^-\hypst}(Q_\triangledown,(\alpha,\alpha^-))} & {\Hilb_{\alpha,\alpha^-}(Q)^{\mathcal{L}(\theta^-)\hypst}/PG(\alpha^-)}
				\arrow["\cong", from=1-1, to=1-2]
				\arrow["\cong", from=3-1, to=3-2]
			\end{tikzcd}\]
		\end{enumerate}
	\end{prop}
	
	To conclude this section, we describe the push-forwards of the line bundles $\mathcal{L}(\theta^\pm)$ under the $PG(\alpha^\pm)$ equivariant isomorphisms $\Phi: \smash{\Hilb^{\alpha^+,\alpha}(Q)} \to \Gr^\alpha(P^+)$ and $\Psi: \Hilb_{\alpha,\alpha^-}(Q) \to \Gr_\alpha(I^-)$.
	
	We consider the Grassmannian $\Gr^\alpha(P^+)$. Let $\mathcal{Q}_i$ be the rank $\alpha_i$ bundle which arises as the pull-back of the universal quotient bundle on $\Gr^{\alpha_i}((P^+)_i)$ under
	\[
		\Gr^\alpha(P^+) \to \prod_{j \in Q_0} \Gr^{\alpha_j}((P^+)_j) \to \Gr^{\alpha_i}((P^+)_i),
	\]
	the composition of the closed embedding and the projection to the $i$\textsuperscript{th} factor. The pull-back of $\mathcal{Q}_i$ under the composition
	\[
		R(Q^{\alpha^+},(1,\alpha))^0 \xto{}{\pi^+} \Hilb^{\alpha^+,\alpha}(Q) \xto{}{\Phi} \Gr^\alpha(P^+)
	\]
	is the trivial bundle with fiber $V_i$ with the action of $G(1,\alpha)$ on the fiber of $(M,A)$ given by $(t,g)\cdot v = t^{-1}g_iv$. Thus $L_i^+ := (\Phi\pi^+)^*\mathcal{L}_i$ is the trivial line bundle with $G(1,\alpha)$-linearization given by the character $G(1,\alpha) \to k^\times,\ (t,g) \mapsto t^{-\alpha_i}\det(g_i)$. This corresponds to the stability parameter $\alpha_ie_0 - e_i \in \Z \times \Z^{Q_0}$. From the definition of $\eta^+$, we see that
	\[
		L(\eta^+) = \bigotimes_{i \in Q_0^+} (L_i^+)^{\otimes (N-\theta_i)} \otimes \bigotimes_{i \notin Q_0^+} (L_i^+)^{\otimes (-\theta_i)}
	\]
	In a similar fashion, on $\Gr_\alpha(I^-)$ let $\mathcal{U}_i$ be the rank $\alpha_i$ bundle arising as the pull-back of the universal sub-bundle under
	\[
		\Gr_\alpha(I^-) \to \prod_{j \in Q_0} \Gr_{\alpha_j}((I^-)_j) \to \Gr_{\alpha_i}((I^-)_i)
	\]
	and let $L_i^- := (\pi^-\Psi)^*\det(\mathcal{U}_i^\vee)$. We can see that
	\[
		L(\eta^-) = \bigotimes_{i \in Q_0^-} (L_i^-)^{\otimes (N+\theta_i)} \otimes \bigotimes_{i \notin Q_0^-} (L_i^-)^{\otimes \theta_i}.
	\]
	This shows:
	
	\begin{lem} \label{l:line_bundles_Gr}
		Let $N \gg 0$.
		\begin{enumerate}
			\item The push-forward $\Phi_*\mathcal{L}(\theta^+)$ under $\Phi: \smash{\Hilb^{\alpha^+,\alpha}(Q)} \smash{\xto{}{\cong}} \Gr^\alpha(P^+)$ equals
			\[
				\mathcal{L}_{\theta^+} :=  \bigotimes_{i \in Q_0^+} \det(\mathcal{Q}_i)^{\otimes (N-\theta_i)} \otimes \bigotimes_{i \notin Q_0^+} \det(\mathcal{Q}_i)^{\otimes (-\theta_i)}.
			\]
			\item The push-forward $\Psi_*\mathcal{L}(\theta^-)$ under $\Psi: \Hilb_{\alpha,\alpha^-}(Q) \smash{\xto{}{\cong}} \Gr_\alpha(I^-)$ equals
			\[
				\mathcal{L}_{\theta^-} := \bigotimes_{i \in Q_0^-} \det(\mathcal{U}_i^\vee)^{\otimes (N+\theta_i)} \otimes \bigotimes_{i \notin Q_0^-} \det(\mathcal{U}_i^\vee)^{\otimes \theta_i}.
			\]

		\end{enumerate}

	\end{lem}

	\begin{ex} \label{e:ex1.4}
		Recall Examples \ref{e:ex1.1}, \ref{e:ex1.2}, and \ref{e:ex1.3}. We have $\theta = (na_1,\ldots,na_m,-|a|)$ and therefore $Q_0^+ = \{q_1,\ldots,q_m\}$ and $Q_0^- = \{s\}$. 
		
		\noindent
		The quiver Grassmannian $\Gr^\alpha(P^+)$ identifies with $\Gr^n(k^m)$. On the latter, there is the universal quotient bundle $\mathcal{Q}$ of $k^m \otimes \OO$ of rank $n$. The quotient bundles on $\Gr^\alpha(P^+)$ are $\mathcal{Q}_{q_i} \cong \OO$ for $i=1,\ldots,m$ and $\mathcal{Q}_s \cong \mathcal{Q}$ for the source $s$. Therefore
		\[
			\mathcal{L}_{\theta^+} \cong \det(\mathcal{Q})^{\otimes |a|}.
		\]
		On the other hand, $\Gr_\alpha(I^-)$ is isomorphic to $(\P^{n-1})^m$. On $(\P^{n-1})^m$, we have the subbundles $\OO(0,\ldots,-1,\ldots,0) = \OO_{\P^{n-1}} \boxtimes \ldots \boxtimes \OO_{\P^{n-1}}(-1) \boxtimes \ldots \boxtimes \OO_{\P^{n-1}}$ of $k^n \otimes \OO$, with a $-1$ in the $i$\textsuperscript{th} position. This bundle identifies with $\UU_{q_i}$ on $\Gr_\alpha(I^-)$. The subbundle $\UU_s$ is $k^n \otimes \OO$, the trivial rank $n$ bundle. We get 
		\[
			\mathcal{L}_{\theta^-} \cong \OO(na_1,\ldots,na_m) = \OO(a_1,\ldots,a_m)^{\otimes n}.
		\]
	\end{ex}
	
	As in the example of the linearly oriented quiver of type $A_n$ the stability parameter is trivial and hence the semi-stable moduli space will be a point, it is not so relevant to determine the line bundles on the quiver Grassmannians in these cases.
	
	\section{Maps to non-commutative Hilbert schemes} \label{s:maps_to_NCHS}
	
	Let $Q$ be an acyclic quiver, $\alpha$ a dimension vector and $\theta$ a stability parameter such that $\theta(\alpha) = 0$. 
	We define two maps
	\begin{align*}
		f^+: R(Q,\alpha) \to R(Q^\triangledown,(\alpha^+,\alpha)),\ & M \mapsto (M,\id) \\
		f^-: R(Q,\alpha) \to R(Q_\triangledown,(\alpha,\alpha^-)),\ & M \mapsto (M,\id),
	\end{align*}
	where $\id = (\id_{V_i})_i \in \smash{\bigoplus_{i \in Q_0^\pm}} \End_k(V_i)$.
	The maps are clearly algebraic and they are $G(\alpha)$-equivariant when we let $G(\alpha)$ act on the right-hand sides by the morphisms of algebraic groups
	\begin{align*}
		d^+: G(\alpha) &\to G(\alpha^+) \times G(\alpha) &
		d^-: G(\alpha) &\to G(\alpha) \times G(\alpha^-) \\
		g &\mapsto ((g_i)_{i \in Q_0^+},g) &
		g &\mapsto (g,(g_i)_{i \in Q_0^-}).
	\end{align*}
	They give rise to morphisms $d^+: PG(\alpha) \to PG(\alpha^+,\alpha)$ and $d^-: PG(\alpha) \to PG(\alpha,\alpha^-)$.
	\pagebreak[3]
	
	\begin{lem} \label{l:saturation}
		\leavevmode
		\begin{enumerate}
			\item The $PG(\alpha^+,\alpha)$-saturation of $f^+(R(Q,\alpha)^{\theta\hypsst})$ equals $R(Q^\triangledown,(\alpha^+,\alpha))^{\theta^+\hypsst}$ and the same is true for the stable loci. The map $f^+$ induces isomorphisms
			\[\begin{tikzcd}
				{M^{\theta\hypsst}(Q,\alpha)} & {M^{\theta^+\hypsst}(Q^\triangledown,(\alpha^+,\alpha))}  \\[-2em]
				{\text{\rotatebox{90}{$\sub$}}} & {\text{\rotatebox{90}{$\sub$}}} \\[-2em]
				{M^{\theta\hypst}(Q,\alpha)} & {M^{\theta^+\hypst}(Q^\triangledown,(\alpha^+,\alpha))}
				\arrow["\cong", from=1-1, to=1-2]
				\arrow["\cong", from=3-1, to=3-2]
			\end{tikzcd}\]
			
			\item The $PG(\alpha,\alpha^-)$-saturation of $f^-(R(Q,\alpha)^{\theta\hypsst})$ equals $R(Q_\triangledown,(\alpha,\alpha^-))^{\theta^-\hypsst}$ and the same is true for the stable loci. The map $f^+$ induces isomorphisms
			\[\begin{tikzcd}
				{M^{\theta\hypsst}(Q,\alpha)} & {M^{\theta^-\hypsst}(Q^\triangledown,(\alpha,\alpha^-))} \\[-2em]
				{\text{\rotatebox{90}{$\sub$}}} & {\text{\rotatebox{90}{$\sub$}}} \\[-2em]
				{M^{\theta\hypsst}(Q,\alpha)} & {M^{\theta^-\hypsst}(Q^\triangledown,(\alpha,\alpha^-))} 
				\arrow["\cong", from=1-1, to=1-2]
				\arrow["\cong", from=3-1, to=3-2]
			\end{tikzcd}\]
		\end{enumerate}
	\end{lem}

	\begin{proof}
		The statements about the saturations of the images follow directly from \Cref{p:theta+-stability}. The inverse of the map $M^{\theta\hypsst}(Q,\alpha) \to M^{\theta^+\hypsst}(Q^\triangledown,(\alpha^+,\alpha))$ is induced by
		\[
			R(Q^\triangledown,(\alpha^+,\alpha))^{\theta^+\hypsst} \to R(Q,\alpha)^{\theta\hypsst}
		\]
		which is defined by $(M,A) \mapsto M$.
	\end{proof}
	
	We now define
	\begin{align*}
		\tilde{\phi}: R(Q,\alpha)^{\theta\hypsst} \to \Hilb^{\alpha^+,\alpha}(Q),\ & M \mapsto [M,\id] \\
		\tilde{\psi}: R(Q,\alpha)^{\theta\hypsst} \to \Hilb_{\alpha,\alpha^-}(Q),\ & M \mapsto [M,\id] 
	\end{align*}
	These maps are well-defined and algebraic. They are $G(\alpha)$-equivariant with respect to the projections $G(\alpha) \to G(\alpha^\pm)$ and therefore $PG(\alpha)$-equivariant with respect to the induced morphism $PG(\alpha) \to PG(\alpha^\pm)$.
	
	The composition $\Phi \tilde{\phi}$ is $\phi: R(Q,\alpha)^{\theta\hypsst} \to \Gr^\alpha(P^+)$ and $\Psi \tilde{\psi}$ is $\psi: R(Q,\alpha)^{\theta\hypsst} \to \Gr_\alpha(I^-)$. This implies: 
		
	\begin{prop}	
		The maps $\phi$ and $\psi$ are well-defined, algebraic and $PG(\alpha)$-equi\-variant with respect to $PG(\alpha) \to PG(\alpha^\pm)$.
	\end{prop}

	\begin{ex} \label{e:ex2.4}
		The maps $R(Q,\alpha) \to \Gr_\alpha(I) \cong \Omega \sub \GL(V)/H$ and $R(Q,\alpha) \to \Gr^\alpha(P) \cong \Upsilon \sub \GL(V)/H$ are equivariant with respect to $PG(\alpha)$; the action of $PG(\alpha)$ is described in \Cref{e:ex2.3}.
	\end{ex}

	\section{The correspondence} \label{s:correspondence}
	
	We are now able to state and prove the two versions of the Gelfand--MacPherson correspondence for quiver moduli. First, we state the version for non-commutative Hilbert schemes.

	\begin{thm} \label{t:correspondence1}
		We have
		\begin{align*}
			\tilde{\phi}(R(Q,\alpha)^{\theta\hypsst}) &= \Hilb^{\alpha^+,\alpha}(Q)^{\mathcal{L}(\theta^+)\hypsst} \\
			\tilde{\phi}(R(Q,\alpha)^{\theta\hypst}) &= \Hilb^{\alpha^+,\alpha}(Q)^{\mathcal{L}(\theta^+)\hypst}
		\end{align*}
		and
		\begin{align*}
			\tilde{\psi}(R(Q,\alpha)^{\theta\hypsst}) &= \Hilb_{\alpha,\alpha^-}(Q)^{\mathcal{L}(\theta^-)\hypst} \\
			\tilde{\psi}(R(Q,\alpha)^{\theta\hypst}) &= \Hilb_{\alpha,\alpha^-}(Q)^{\mathcal{L}(\theta^-)\hypst}.
		\end{align*}
		Moreover, $\tilde{\phi}$ and $\tilde{\psi}$ induce isomorphisms
		\[\begin{tikzcd}
			{\Hilb^{\alpha^+,\alpha}(Q)^{\mathcal{L}(\theta^+)\hypsst}/\!\!/PG(\alpha^+)} & {M^{\theta\hypsst}(Q,\alpha)} & {\Hilb_{\alpha,\alpha^-}(Q)^{\mathcal{L}(\theta^-)\hypsst}/\!\!/PG(\alpha^-)} \\[-2em]
			{\text{\rotatebox{90}{$\sub$}}} & {\text{\rotatebox{90}{$\sub$}}} & {\text{\rotatebox{90}{$\sub$}}} \\[-2em]
			{\Hilb^{\alpha^+,\alpha}(Q)^{\mathcal{L}(\theta^+)\hypst}/PG(\alpha^+)} & {M^{\theta\hypst}(Q,\alpha)} & {\Hilb_{\alpha,\alpha^-}(Q)^{\mathcal{L}(\theta^-)\hypst}/PG(\alpha^-).}
			\arrow["\cong", swap, from=1-2, to=1-1]
			\arrow["\cong", from=1-2, to=1-3]
			\arrow["\cong", swap, from=3-2, to=3-1]
			\arrow["\cong", from=3-2, to=3-3]
		\end{tikzcd}\]
	\end{thm}
	
	\begin{proof}
		The equalities for the images of the (semi\nobreakdash-)stable loci under $\phi^\pm$ follows from \Cref{p:quotients_Hilbert_schemes} and \Cref{l:saturation}.
		
		\noindent
		To show that $\tilde{\phi}$ induces an isomorphism $M^{\theta\hypsst}(Q,\alpha) \smash{\xto{}{\cong}} \smash{\Hilb^{\alpha^+,\alpha}(Q)^{\mathcal{L}(\theta^+)\hypsst}/\!\!/PG(\alpha^+)}$, we consider the following big diagram:
		\[\begin{tikzcd}
			&[-1em] {R(Q,\alpha)^{\theta\hypsst}} &[-1.2em] {R(Q^\triangledown,(\alpha^+,\alpha))^{\theta^+\hypsst}} &[-5em] \subseteq &[-3em] {R(Q^{\alpha^+},(1,\alpha))^0} \\
			&& {\Hilb^{\alpha^+,\alpha}(Q)^{\mathcal{L}(\theta^+)\hypsst}} & \subseteq & {\Hilb^{\alpha^+,\alpha}(Q)} \\
			{M^{\theta\hypsst}(Q,\alpha)} & {M^{\theta^+\hypsst}(Q^\triangledown,(\alpha^+,\alpha))} & {\Hilb^{\alpha^+,\alpha}(Q)^{\mathcal{L}(\theta^+)\hypsst}/\!\!/PG(\alpha^+)}
			\arrow[from=1-2, to=3-1]
			\arrow["f^+",from=1-2, to=1-3]
			\arrow[from=1-3, to=3-2, start anchor={[xshift=-5ex]}]
			\arrow["\cong", from=3-1, to=3-2]
			\arrow["\tilde{\phi}",from=1-2, to=2-3, crossing over]
			\arrow[from=1-3, to=2-3]
			\arrow["\pi^+",from=1-5, to=2-5]
			\arrow[from=2-3, to=3-3]
			\arrow["\cong", from=3-2, to=3-3]
		\end{tikzcd}\]
		The left-hand parallelogram is commutative by \Cref{l:saturation}. Commutativity of the triangle in the middle follows from \Cref{p:quotients_Hilbert_schemes}. The claim is proved. 
		
		\noindent		
		In a similar vein, we see that the induced isomorphism restricts to an isomorphism $M^{\theta\hypst}(Q,\alpha) \smash{\xto{}{\cong}} \smash{\Hilb^{\alpha^+,\alpha}(Q)^{\mathcal{L}(\theta^+)\hypst}/PG(\alpha^+)}$. The claims for $\tilde{\psi}$ can be proved analogously.
	\end{proof}
	
	Now to the version using quiver Grassmannians. Recall the definition of the line bundles $\mathcal{L}_{\theta^\pm}$ in \Cref{l:line_bundles_Gr}.
	
	\begin{thm} \label{t:correspondence2}
		We have
		\begin{align*}
			\phi(R(Q,\alpha)^{\theta\hypsst}) &= \Gr^\alpha(P^+)^{\LL_{\theta^+}\hypsst} &
			\psi(R(Q,\alpha)^{\theta\hypsst}) &= \Gr_\alpha(I^-)^{\LL_{\theta^-}\hypsst} \\
			\phi(R(Q,\alpha)^{\theta\hypst}) &= \Gr^\alpha(P^+)^{\LL_{\theta^+}\hypst} &
			\psi(R(Q,\alpha)^{\theta\hypst}) &= \Gr_\alpha(I^-)^{\LL_{\theta^-}\hypst}.
		\end{align*}
		Moreover, $\phi$ and $\psi$ induce isomorphisms
		\[\begin{tikzcd}
			{\Gr^\alpha(P^+)^{\LL_{\theta^+}\hypsst}/\!\!/PG(\alpha^+)} & {M^{\theta\hypsst}(Q,\alpha)} & {\Gr_\alpha(I^-)^{\LL_{\theta^-}\hypsst}/\!\!/PG(\alpha^-)} \\[-2em]
			{\text{\rotatebox{90}{$\sub$}}} & {\text{\rotatebox{90}{$\sub$}}} & {\text{\rotatebox{90}{$\sub$}}} \\[-2em]
			{\Gr^\alpha(P^+)^{\LL_{\theta^+}\hypst}/PG(\alpha^+)} & {M^{\theta\hypst}(Q,\alpha)} & {\Gr_\alpha(I^-)^{\LL_{\theta^-}\hypst}/PG(\alpha^-)}
			\arrow["\cong", from=1-2, to=1-3]
			\arrow["\cong", from=3-2, to=3-3]
			\arrow["\cong", swap, from=1-2, to=1-1]
			\arrow["\cong", swap, from=3-2, to=3-1]
		\end{tikzcd}\]
	\end{thm}

	\begin{proof}
		Follows from \Cref{p:identification_NCHS_QGr}, \Cref{l:equivariance}, and \Cref{t:correspondence1}.
	\end{proof}
	
	\begin{ex}
		For the $m$-subspace quiver $Q$ the dimension vector $\alpha = (1,\ldots,1,n)$ and $\theta = (na_1,\ldots,na_m,-|a|)$, we obtain isomorphisms
		\[
			\Gr^n(k^m)^{\sst}/\!\!/(k^\times)^m \xot{}{\cong} M^{\theta\hypsst}(Q,\alpha) \xto{}{\cong} ((\P^{n-1})^m)^{\sst}/\!\!/\PGL_n(k)
		\]
		where semi-stability on $\Gr^n(k^m)$ is with respect to $\mathcal{L}_{\theta^+} = \det(\mathcal{Q})^{\otimes |a|}$ and semi-stability on $(\P^{n-1})^m$ is with respect to $\mathcal{L}_{\theta^-} = \OO(a_1,\ldots,a_m)^{\otimes n}$.
	\end{ex}

	\begin{rem} \label{r:hu_kim}
		In \cite{HK:13}, Hu and Kim consider the following situation. Let $Q$ be a bipartite quiver, which means that every vertex of $Q$ is either a source or a sink. Let $I, J \sub Q_0$ be the sets of sources and sinks respectively. Let $\alpha$ be a dimension vector for $Q$ and $\theta$ be a stability parameter such that $\theta(\alpha) = 0$. Assume that $\theta_i > 0$ for every $i \in I$, while $\theta_j < 0$ for every $j \in J$. Fix vector spaces $V_i$ of dimension $\alpha_i$ for every $i \in Q_0$. We get
		\begin{align*}
			P^+ &= \bigoplus_{i \in I} P(i) \otimes V_i \\
			I^- &= \bigoplus_{j \in J} I(j) \otimes V_j.
		\end{align*}
		For $i \in I$, we have $P^+_i = k\epsilon_i \otimes V_i$, so 
		\[
			\Gr^\alpha(P^+) \cong \prod_{j \in J} \Gr^{\alpha_j}(P^+_j) \cong \prod_{j \in J} \Gr^{\alpha_j}\Big(\bigoplus_{\substack{a \in Q_1\\t(a) = j}} V_{s(a)}\Big).
		\]
		This variety carries an action of $G(\alpha^+) = \prod_{i \in I} \GL(V_i)$ which descends to an action of $PG(\alpha)$. We also obtain an ample linearization
		\[
			\mathcal{L}_{\theta^+} = \bigotimes_{j \in J} \det(\mathcal{Q}_j)^{-\theta_j}
		\]
		of the action. The semi-stable quotient is isomorphic to $M^{\theta\hypsst}(Q,\alpha)$. Similarly for the injective representation. There, $I^-_j = k\epsilon_j^* \otimes V_j$ and thus
		\[
			\Gr_\alpha(I^-) \cong \prod_{i \in I} \Gr_{\alpha_i}(I^-_i) \cong \prod_{i \in I} \Gr^{\alpha_i}\Big(\bigoplus_{\substack{a \in Q_1\\s(a) = i}} V_{t(a)}\Big).
		\]
		On this variety, $G(\alpha^-) = \smash{\prod_{j \in J}} \GL(V_j)$ acts and the action descends to an action of $PG(\alpha^-)$. It is linearized by 
		\[
			\mathcal{L}_{\theta^-} = \bigotimes_{i \in I} \det(\mathcal{U}_i^\vee)^{\theta_i}.
		\]
		Also the semi-stable quotient of this Grassmannian is isomorphic to $M^{\theta\hypsst}(Q,\alpha)$. We recover \cite[Thm.\ 1.1]{HK:13}.
	\end{rem}
	
	\begin{rem}
		Let us point out that, for an acyclic quiver, the identification of $M^{\theta\hypsst}(Q,\alpha)$ with a GIT quotient of a projective variety implies that $M^{\theta\hypsst}(Q,\alpha)$ is projective without using the theorem of Le Bruyn and Procesi \cite{LP:90}.
	\end{rem}
	
	\bibliographystyle{abbrv}
	\bibliography{Literature}
\end{document}